%% file: paper-surjectivity.tex
\documentclass[a4paper]{amsart}

\input{preamble}

\title{On the surjectivity conjectures of Dupont and Monod}

\author{Alexander Kupers}
\address{Department of Computer and Mathematical Sciences, University of Toronto Scarborough, 1265 Military Trail, Toronto, ON M1C 1A4, Canada}
\email{a.kupers@utoronto.ca}

\author{Daniil Rudenko}
\address{Department of Mathematics,
University of Chicago,
5734 S. University Avenue, 
Chicago, IL 60637, USA}
\email{rudenkodaniil@uchicago.edu}

\author{Ismael Sierra}
\address{School of Mathematics and Statistics, University of Glasgow, University Place, Glasgow, G12 8QQ, United Kingdom.}
\email{ismael.sierradelrio@glasgow.ac.uk}

\date{\today}

\begin{document}
	
\begin{abstract}In this paper we apply our work on polylogarithmic cocycles representing the Borel classes to establish that the comparison map from bounded continuous cohomology to continuous cohomology is surjective for connected semisimple Lie groups with finite centre. This proves conjectures of Dupont and Monod.
\end{abstract}
	
\maketitle

\vspace{-.75cm} 

\tableofcontents

\vspace{-.75cm} 

\section{Introduction}

The cohomology $H^*(G;\bb{R})$ of a group $G$ with real coefficients can be computed as the cohomology of the inhomogeneous bar complex $C^*(G;\bb{R}) \coloneq \rm{Fun}(G^p,\bb{R})$ of $\bb{R}$-valued functions with differential
\[(df)([g_0|\compactcdots|g_{p}]) = f([g_1|\compactcdots|g_p])+ \sum_{i=1}^{p}(-1)^{i} f([g_0|\compactcdots|g_{i-1}g_i|\cdots|g_p])+(-1)^{p+1} f([g_0|\compactcdots|g_{p-1}]).\] 
If $G$ is a topological group, we can define variants of group cohomology by imposing additional conditions on these functions: taking $C^*_\rm{cts}(G;\bb{R}) \subseteq  C^*(G;\bb{R})$ to be the subcomplex of continuous functions we obtain \emph{continuous cohomology}, and taking $C^*_\rm{bd}(G;\bb{R}) \subseteq C^*(G;\bb{R})$ to be subcomplex of bounded functions we obtain \emph{bounded cohomology}. We can also intersect these to get \emph{bounded continuous cohomology}, fitting in a commutative square of natural comparison maps induced by the inclusions:
\[\begin{tikzcd} H^*_\rm{bd,cts}(G;\bb{R}) \rar \dar[swap]{\circled{1}} & H^*_\rm{bd}(G;\bb{R}) \dar{\circled{2}} \\[-5pt]
H^*_\rm{cts}(G;\bb{R}) \rar{\circled{3}} & H^*(G;\bb{R}).\end{tikzcd}\]
Inspired by a question of Dupont (see Remark \eqref{enum:dupont}), Monod asked  in \cite[Problem A']{MonodICM} whether for a semisimple Lie group $G$ with finite centre, the map $\circled{1}$ is an isomorphism: we prove it is a surjection.

\begin{atheorem} \label{thm:monod}
    Let $G$ be a connected semisimple Lie group $G$ with finite centre. Then the following comparison map is surjective
    \[\circled{1} \colon H^*_\rm{bd,cts}(G;\R) \lra H^*_\rm{cts}(G;\R).\]
\end{atheorem}

\begin{remark*}\,
\begin{enumerate}[(i)]
    \item \cref{thm:monod} is false without the hypothesis that $G$ has finite centre: take $G = \widetilde{\SL}_2(\bb{R})$ \cite[Example 9.3.11]{MonodBook}.
    \item \label{enum:dupont} \cref{thm:monod} implies the image of $\circled{3}$ is contained in the image of $\circled{2}$, a conjecture implicit in a question of Dupont \cite[Remark 3]{DupontBounds} about integrals over geodesic simplices. 
    \item Previous work on \cref{thm:monod} includes \cite{DupontBounds,MonodBook,BMlattices,BMlatticeserrata,MonodSLn,CMPSC,HartnickOttPDE,BBI,Mengual,KastenholzSroka}. Moreover, Hartnick and Ott proved it for hermitian Lie groups \cite{HartnickOtt} , Goncharov proved \cref{thm:monod} for $\SL_3(\bb{C})$ \cite[Section 9]{Gon95b}, and for $\SL_4(\bb{C})$ it can be deduced from work of Goncharov and Rudenko \cite{GR18}.
    \item The comparison map of \cite[Problem A']{MonodICM} was proven to be an isomorphism in the nonarchimedean setting in \cite{MonodFlamtes}; in this setting surjectivity is trivial.
\end{enumerate}
\end{remark*}

Let us outline the strategy of the proof. The key result is the case $\SL_n(\C)$, whose continuous cohomology is generated by the Borel classes in degrees $\ge 3$. We will first prove it in this case, establishing \cite[Conjecture 9.3.8]{MonodBook}. Since the property of admitting a bounded continuous representative is closed under cup products, it suffices prove that the Borel classes have this property. We do so by combining \cite[Theorem D, E]{KRS2} with an extension criterion due to Goncharov (\cref{thm:extension-criterion}). In the former, we gave cocycle for the Borel class $\rm{Bo}_{n,2n-1} \in H^{2n-1}(\GL_n(\bb{C});\bb{R})$ that, away from a subset contained in a finite union of submanifolds codimension one in $\GL_n(\C)^{2n-1}$, is given by finite linear combination of Hodge correlators \cite{Gon19,Mal20,KT24} evaluated at rational functions of the entries of the matrices. Since Hodge correlators are bounded continuous functions for $n \geq 2$, this is an essentially bounded measurable cocycle and provides by mollification a bounded continuous cocycle. The Borel class is stable and Goncharov's criterion produces from our previous cocycle a cocycle that represents the analogous class $\rm{Bo}_{m+n,2n-1} \in H^{2n-1}(\GL_{m+n}(\bb{C});\bb{R})$ for $m > 0$, which is essentially bounded measurable by a similar argument.

The deduction of \cref{thm:monod} from the special case $\SL_n(\bb{C})$ then amounts to mere ``classical'' results in the theory of semisimple Lie groups and their cohomology. We first reduce to the case that $G$ is simple. There are three classes of simple Lie groups:
\begin{enumerate}[(A)]
\item compact simple Lie groups, 
\item complex simple Lie groups, and 
\item noncompact real simple Lie groups.
\end{enumerate} 
Those in class (A) have trivial continuous cohomology, so there is nothing to prove. We first deal with case (B): from classical computations we deduce that for noncompact complex simple Lie groups $G$ there exist nontrivial complex representations $G \to \SL(V)$ so that the induced maps on continuous cohomology are jointly surjective. We next deal with case (C): by combining work of Hartnick and Ott with an understanding of the map induced on continuous cohomology by the inclusion of a noncompact real simple Lie group into its complexification, we reduce this case to case (B). This completes the proof of \cref{thm:monod}.

\subsection{Applications} \cref{thm:monod} has several applications that we want to quickly point out:
\begin{enumerate}[(a)]
    \item \cref{thm:monod} yields an alternative proof of a conjecture due to Gromov which says that locally symmetric spaces of non-compact type have finite simplicial volume (see \cite[p.~61]{Karlsson}). This was previously proven by Lafont and Schmidt \cite{LafontSchmidt} .
    \item \cref{thm:monod} and \cite[Theorem 2]{BurgerIozzi} combine to yield the following: given $\Gamma \subseteq G$ a torsion-free and discrete subgroup of a semisimple Lie group with finite centre and associated locally symmetric space $\Gamma \backslash G/K$, any class in the image of $H^*_\rm{cts}(G;\bb{R}) \to H^*(G;\bb{R})$ pulls back along the map $\Gamma \backslash  G/K \to \rm{B}\Gamma \to \rm{B}G$ to a class that is represented in $H^*_\rm{dR}(\Gamma \backslash  G/K)$ by a closed bounded form. Thus, for example, the Borel classes for a torsion-free subgroup $\Gamma \subseteq \SL_n(\bb{Z})$ admit representatives on $\Gamma \backslash \SL_n(\bb{R})/\rm{SO}_n$ that are closed bounded forms.
    \item \cref{thm:monod} can be used to construct explicit cocycles for other cohomology classes or characteristic classes. Examples include characteristic classes of flat vector bundles or flat manifold bundles, or continuous cohomology classes for diffeomorphism groups \cite{Costes}.
\end{enumerate}

\subsection*{Acknowledgments} AK acknowledges the support of the Natural Sciences and Engineering Research Council of Canada (NSERC) [funding reference number 512156 and 512250]. DR was supported by NSF grant DMS-2502729. IS acknowledges the University of Toronto. We would like to thank Nicolas Monod and Robin Sroka for helpful discussions.

\section{Continuous cohomology of Lie groups and the van Est isomorphism} In this section we discuss definitions and results on the continuous cohomology of Lie groups. We will also need a substantial amount of classical Lie theory, for which the texts \cite{helgason2001differential,VGO} serve as our main references.

\begin{notation} We will often see the cohomology with real coefficients of a Lie group $G$ \emph{as a space}. To avoid confusion, we will write $H^*_\rm{sing}(G;\bb{R})$.
\end{notation}

\subsection{Continuous, smooth, and measurable cohomology}\label{sec:continuous-cohomology-etc} As in the introduction, for a group $G$ we consider the cochain complex $C^*(G;\bb{R}) \coloneq \rm{Fun}(G^p,\bb{R})$  with differential
\[(df)([g_0|\compactldots|g_{p}]) = f([g_1|\compactldots|g_p])+ \sum_{i=0}^{p-1} (-1)^{i} f([g_0|\compactldots|g_{i-1}g_i|\ldots|g_p])+(-1)^{p+1} f([g_0|\compactldots|g_{p-1}]).\] 
For a Lie group $G$ there are then variants $C^*_\rm{sm}(G;\bb{R}) \subseteq C^*_\rm{cts}(G;\bb{R}) \subseteq  C^*(G;\bb{R})$ of smooth, respectively continuous functions. There is also a measurable variant $C^*_\rm{m}(G;\bb{R})$, see \cite{AustinMoore}. As continuous functions are measurable, there is a map $C^*_\rm{cts}(G;\bb{R}) \to C^*_\rm{m}(G;\bb{R})$. The induced maps on cohomology are all isomorphisms
\[H^*_\rm{sm}(G;\bb{R}) \overset{\cong}\lra H^*_\rm{cts}(G;\bb{R}) \overset{\cong}\lra H^*_\rm{m}(G;\bb{R}).\]
Because of this, we will generally not distinguish between these and write $H^*_\rm{cts}(G;\bb{R})$ for all. We will need that these cohomology groups share properties with usual group cohomology, e.g.~there exist transfers for finite index subgroups and inner automorphisms act trivially; these can be proven with the same arguments. 

\subsection{The van Est isomorphism} The computation of the continuous cohomology of Lie groups is made straightforward using the van Est isomorphism. To state it, we introduce the following:

\begin{definition}If $\fr{k} \subset \fr{g}$ is a Lie subalgebra, the \emph{relative Lie algebra cochains} $C^*(\fr{g},\fr{k})$ are defined as
\[
   C^*(\fr{g},\fr{k}) \coloneq \left\{\varphi \in \rm{Hom}(\Lambda^*(\fr{g}/\fr{k}),\bb{R}) \,\middle|\,  \parbox{6cm}{\centering $\sum_{i=1}^p\varphi(\overline{X}_1,\dots,\overline{[X,X_i]},\dots,\overline{X}_p)=0$ $\forall X \in \fr{k}$, $\forall \overline{X}_1, \dots, \overline{X}_p \in \fr{g}/\fr{k} $}\right\}, 
\]
with differential given by
\[d\varphi(\overline{X}_0,\ldots,\overline{X}_p) = \sum_{i<j} (-1)^{i+j} \varphi(\ol{[X_i,X_j]},\overline{X}_0,\ldots,\widehat{\overline{X}_i},\ldots,\widehat{\overline{X}_j},\ldots,\ol{X_p}).\]
The \emph{relative Lie algebra cohomology} $H^*(\fr{g},\fr{k})$ is the cohomology of the complex $C^*(\fr{g},\fr{k})$.
\end{definition}

This is natural in maps of pairs of a Lie algebra and a Lie subalgebra. It vanishes if $\fr{g} = \fr{k}$ and when $\fr{k}=0$ recovers the usual Chevalley--Eilenberg complex computing Lie algebra cohomology. It can be used to compute continuous cohomology \cite{vanEst}. 

\begin{theorem}[van Est] \label{thm:van-est} If $G$ is a connected Lie group and $K \subseteq G$ is a maximal compact subgroup, then there is an isomorphism
\[H^*_\rm{cts}(G;\bb{R}) \cong H^*(\fr{g},\fr{k}),\]
where $\fr{g}, \fr{k}$ are the Lie algebras of $G$ and $K$ respectively. 
\end{theorem}

In \cite[Section 5.4.3]{KRS2} we used a specific cochain map $C^*(\fr{g},\fr{k}) \overset{\cong}\lra C^*_\rm{cts}(G;\bb{R})$ inducing an isomorphism, due to Dupont \cite[\S 1]{DUPONT1976233}. Here it is more convenient to use what is essentially an inverse: given a connected Lie group $G$ with maximal compact subgroup $K$, there is a zigzag of maps of chain complexes
\[ C^p_\rm{sm}(G;\bb{R}) \overset{\inc}\longleftarrow  C^p_\rm{sm}(G;\bb{R})^K \lra C^p(\fr{g},\fr{k})\]
where $C^p_\rm{sm}(G;\bb{R})^K$ is given by the subspace of those functions that are invariant under the $p+1$ commuting actions of $K$ given in \cite[(4)]{Salazar}, and the right map is given by
\begin{align*}C^p_\rm{sm}(G;\bb{R})^K &\lra C^p(\fr{g},\fr{k}) \\
f &\longmapsto \left((\ol{X}_1,\ldots,\ol{X}_p) \longmapsto \sum_{\sigma \in \fr{S}_p} \rm{sign}(\sigma)\,D_{\ol{X}_{\sigma(1)}} \cdots D_{\ol{X}_{\sigma(p)}} f\right) \end{align*}
where for $\ol{X} \in \fr{g}/\fr{k}$ and $f \in C^p_{\rm{sm}}(G)$, $D_{\ol{X}} f \in C^{p-1}(G)$ is given by lifting $\ol{X}$ to $X \in \fr{g}$, extending it to a left-invariant vector field on $G$, and sending $(g_1,\ldots,g_{p-1})$ to the directional derivative along $X$ at the identity of the smooth function $g \mapsto f(g,g_1,\ldots,g_{p-1})$. These maps induces isomorphisms on cohomology, yielding an isomorphism as in \cref{thm:van-est}
\[\rm{vE}_{G,K} \colon H^*_\rm{cts}(G;\bb{R}) \overset{\cong}\lra H^*(\fr{g},\fr{k}).\]
A modern reference is \cite[Theorem 5.10]{Salazar} (Lie groups and Lie algebras are a special case of Lie groupoids and Lie algebroids), but this may also be deduced from \cite{Guichardet} or \cite{Houard}.

It follows from the construction that if $\varphi \colon G \to G'$ is a homomorphism of connected Lie groups and $K \subseteq G$, $K' \subseteq G'$ are maximal compact subgroups such that $\varphi(K) \subseteq K'$, then the following diagram commutes
\[\begin{tikzcd} H^*_\rm{cts}(G';\bb{R}) \rar{\rm{vE}_{G',K'}} \dar[swap]{\varphi^*} &[5pt] H^*(\fr{g}',\fr{k}') \dar{\varphi^*} \\[-5pt]
H^*_\rm{cts}(G;\bb{R}) \rar{\rm{vE}_{G,K}} & H^*(\fr{g},\fr{k}),\end{tikzcd}\]
where the right vertical map is given by induced map on pairs of Lie algebras.

The van Est isomorphism depends a priori on a choice of $K$, but the latter is mostly immaterial so we will later suppress it from the notation: if $K' \subseteq G$ is a different choice of maximal compact subgroup, there is $g \in G$ such that $g K g^{-1} = K'$ and the commutative diagram for the inner automorphism $c_g \colon G \to G$ given by conjugation with $g$ takes the form
\[\begin{tikzcd} H^*_\rm{cts}(G;\bb{R}) \rar{\rm{vE}_{G,K'}} \dar[swap]{c_g^* = \id} &[5pt] H^*(\fr{g},\fr{k}') \dar{c_g^*} \\[-5pt]
H^*_\rm{cts}(G;\bb{R}) \rar{\rm{vE}_{G,K}} & H^*(\fr{g},\fr{k}).\end{tikzcd}\]

\subsection{The compact dual} \label{sec:compact-duals} In the semisimple case we can identify relative Lie algebra cohomology with the cohomology of a compact dual symmetric space using a ``unitary trick''.

\subsubsection{Simple and semisimple Lie algebras} We start with some recollections. A finite-dimensional Lie algebra $\fr{g}$ is \emph{simple} if it is not abelian and its only Lie ideals are $\{0\}$ and $\fr{g}$. It is \emph{semisimple} if it is isomorphic to a direct sum of simple Lie algebras. 

A Lie algebra $\fr{g}$ has an \emph{adjoint representation} given by the Lie algebra homomorphism
\begin{align*} \rm{ad} \colon \fr{g} &\lra \fr{gl}(\fr{g}) \\
X &\longmapsto \rm{ad}_X \coloneq [X,-].\end{align*}
For semisimple $\fr{g}$, the Lie algebra homomorphism $\rm{ad}$ is injective and the \emph{adjoint group} $\rm{Int}(\fr{g})$ is the connected subgroup of $\rm{GL}(\fr{g})$ with Lie algebra given by $\rm{im}(\rm{ad}) \subseteq \fr{gl}(\fr{g})$ \cite[II.\S 5]{helgason2001differential}. It has trivial centre \cite[II.5.3]{helgason2001differential}.

\subsubsection{Compact duals}

Let $G$ be a connected Lie group and let $K \subseteq G$ be a maximal compact subgroup \cite[VI.2.2]{helgason2001differential}. The inclusion $K \subseteq G$ of a maximal compact subgroup induces an inclusion $\fr{k} \subseteq \fr{g}$ of Lie algebras. For semisimple $\fr{g}$ this can be extended uniquely to a \emph{Cartan decomposition} \cite[III.\S 7]{helgason2001differential} \cite[4.3.1]{VGO}, given by a direct decomposition of vector spaces
\[\fr{g} = \fr{k} \oplus \fr{p} \qquad \text{with $\fr{k},\fr{p} \subseteq \fr{g}$}\]
satisfying (a) $[\fr{k},\fr{p}] \subseteq \fr{p}$ and $[\fr{p},\fr{p}] \subseteq \fr{k}$, and (b) the Killing form $(X,Y) \mapsto \rm{tr}(\rm{ad}_X \circ \rm{ad}_Y)$ is negative-definite on $\fr{k}$ and positive-definite on $\fr{p}$. As $\fr{p}$ and $\fr{k}$ are orthogonal complements with respect to the Killing form, one determines the other. Condition (a) is equivalently to  the \emph{Cartan involution} $\theta \colon \fr{g} \to \fr{g}$ given by $\rm{id}_\fr{k} \oplus - \rm{id}_\fr{p}$ is an automorphism of Lie algebras. A Cartan involution $\theta \colon \fr{g} \to \fr{g}$ is induced by an automorphism $\Theta \colon G \to G$ of Lie groups that we will also refer to as the Cartan involution \cite[Theorem 4.3.2]{VGO}.

Given a Cartan decomposition $\fr{g} = \fr{k} \oplus \fr{p}$ we can define a second Lie algebra by taking $\fr{g}_\rm{u} \coloneqq \fr{k} \oplus \fr{p}$ and letting for $x,y \in \fr{k}$ and $u,v \in \fr{p}$ the Lie bracket be given by
\[[x,y]_\rm{u} = [x,y], \quad [x,u]_\rm{u} = [x,u], \quad [u,v]_\rm{u} = -[u,v].\]
This contains $\fr{k}$ as a subalgebra. One think of it as the real Lie subalgebra $\fr{k} \oplus i \frak{p} \subseteq \fr{g} \otimes \bb{C}$ of the complexification of $\fr{g}$. Through this construction, Cartan decompositions are in bijection with compact real forms of the complexification $\fr{g} \otimes \bb{C}$ \cite[4.\S 1.2]{VGO}, Here a \emph{real form} of a complex Lie algebra $\fr{h}_\bb{C}$ is a real Lie subalgebra $\fr{h} \subseteq \fr{h}_\bb{C}$ so that $\fr{h}_\bb{C} = \fr{h} \oplus i \fr{h}$ \cite[III.\S 6]{helgason2001differential}. A real semisimple Lie algebra is \emph{compact} if its Killing form is negative definite, or equivalently if and only if all connected Lie groups with that Lie algebra isomorphic are compact \cite[Corollary II.6.9, Theorem II.6.9]{helgason2001differential} \cite[Proposition 4.1.3]{VGO}.

The identity map on underlying vector spaces then induces an isomorphism
\[\upsilon_{\fr{g},\fr{k}} \colon H^*(\fr{g},\fr{k}) \overset{\cong}\lra H^*(\fr{g}_\rm{u},\fr{k}).\] 
Note that the Cartan involution acts on relative Lie algebra cochains by $(-1)^p$ in degree $p$, so the differential must be trivial, so there is also an isomorphism of relative Lie algebra algebra cochains. Since $\fr{g}_\rm{u}$ is compact, $G_\rm{u} \coloneq \rm{Int}(\fr{g}_\rm{u})$ is a compact connected Lie group: this is the \emph{compact dual group}. The same is true for the connected subgroup $K_\rm{u} \subseteq G_\rm{u}$ corresponding to $\fr{k} \subseteq \fr{g}_\rm{u}$ and the quotient symmetric space $G_\rm{u}/K_\rm{u}$: this is the \emph{compact dual symmetric space}. There is an isomorphism
\[\mu_{\fr{g},\fr{k}} \colon H^*(\fr{g}_\rm{u},\fr{k}) \overset{\cong}\lra H^*(G_\rm{u}/K_\rm{u};\bb{R})\]
inducing by sending relative Lie algebra cochain to the corresponding left $G_\rm{u}$-invariant differential form on the compact dual symmetric space \cite[Theorem 22.1]{chevalley1948cohomology}. We conclude:

\begin{proposition}\label{prop:van-est-symmetric} If $G$ is a connected semisimple Lie group with finite centre, $K \subseteq G$ is a maximal compact subgroup, and $K_\rm{u} \subseteq G_\rm{u}$ are the compact duals, then there is an isomorphism
\[\mu_{\fr{g},\fr{k}} \circ \upsilon_{\fr{g},\fr{k}} \circ \rm{vE}_G \colon H^*_\rm{cts}(G;\bb{R}) \overset{\cong}\lra H^*_\rm{sing}(G_\rm{u}/K_\rm{u};\bb{R}).\]
\end{proposition}

Both $\upsilon_{\fr{g},\fr{k}}$ and $\mu_{\fr{g},\fr{k}}$ are natural in maps of semisimple Lie algebras $\fr{g} \to \fr{g}'$ with compatible Cartan decompositions, in the sense that the following diagram commutes
\[\begin{tikzcd} H^*(\fr{g}',\fr{k}') \dar \rar{\cong}[swap]{\upsilon_{\fr{g}',\fr{k}'}} & H^*(\fr{g}'_\rm{u},\fr{k}') \dar \rar{\cong}[swap]{\mu_{\fr{g}',\fr{k}
}} & H^*_\rm{sing}(G'_\rm{u}/K'_\rm{u};\bb{R}) \dar \\[-5pt]
H^*(\fr{g},\fr{k}) \rar{\cong}[swap]{\upsilon_{\fr{g},\fr{k}}} & H^*(\fr{g}_\rm{u},\fr{k}) \rar{\cong}[swap]{\mu_{\fr{g},\fr{k}}} & H^*_\rm{sing}(G_\rm{u}/K_\rm{u};\bb{R}).\end{tikzcd}\]

\begin{example}The \emph{geometric map} \cite[Section 2]{HartnickOtt} is
\begin{equation}\label{eqn:geometric-map} H^*_\rm{sing}(\rm{B}K_\rm{u};\bb{R}) \overset{q^*}\lra H^*(G_\rm{u}/K_\rm{u};\bb{R}) \xrightarrow{\mu_{\fr{g},\fr{k}}^{-1}} H^*(\fr{g}_\rm{u},\fr{k}) \xrightarrow{\upsilon_{\fr{g},\fr{k}}^{-1}} H^*(\fr{g},\fr{k}) \xrightarrow{\rm{vE}_G^{-1}} H^*_\rm{cts}(G;\bb{R})\end{equation}
where $q \colon G_\rm{u}/K_\rm{u} \to \rm{B}K_\rm{u}$ is induced by $G_\rm{u} \to \ast$. Though this map depends on a maximal compact subgroup $K \subseteq G$, its image is independent of this choice using that maximal compact subgroups are unique up to conjugation and inner automorphisms act trivially on cohomology.\end{example}

\subsection{Reduction to simple Lie algebras}  Our next goal is to reduce the semisimple case to the simple case, and we start in this subsection with with some recollections.

\subsubsection{Reduction} We first prove that whether \cref{thm:monod} holds for a connected semisimple Lie group with finite centre $G$, only depends on the isomorphism class of its Lie algebra $\fr{g}$. A Lie group $G$ has an \emph{adjoint representation} 
\begin{align*} \rm{Ad} \colon G &\lra \rm{GL}(\fr{g}) \\
g &\longmapsto d_e(g(-)g^{-1})\end{align*}
which for connected semisimple Lie groups $G$ with finite centre satisfies \cite[Corollary II.5.2]{helgason2001differential} (i) there is an exact sequence
\begin{equation} \label{eq:adjoint-centre}
    1 \lra Z(G) \lra G \lra G_\rm{ad} \coloneq \rm{im}(\rm{Ad}) \lra 1
\end{equation}
and (ii) there is an identification $\rm{im}(\rm{Ad}) = \rm{Int}(\fr{g}) \subseteq \rm{GL}(\fr{g})$. Using this, we prove:

\begin{proposition}\label{prop:lie-algebra-monod} Suppose two connected semisimple Lie groups with finite centre $G,G'$ have isomorphic Lie algebras. Then \cref{thm:monod} is true for $G$ if and only if it is true for $G'$.
\end{proposition}

\begin{proof} The adjoint representations of $G$ and $G'$ provide homomorphisms onto $\rm{Int}(\fr{g})$ with finite kernel, since $G, G'$ have finite centre by assumption. Thus it suffices to prove the proposition for $G' = G_\rm{ad}$. Using the commutative diagram
\[\begin{tikzcd} H^*_{\rm{bd,cts}}(G_\rm{ad};\bb{R}) \rar \dar & H^*_\rm{cts}(G_\rm{ad};\bb{R}) \dar \\[-5pt]
H^*_{\rm{bd,cts}}(G;\bb{R}) \rar & H^*_\rm{cts}(G;\bb{R})\end{tikzcd}\]
where, by the existence of transfers \cite[Section 8.6]{MonodBook}, the vertical maps are split injective onto the $Z(G)$-invariants. But these are are the entire group since $Z(G)$ is central, so the vertical maps are isomorphisms.\end{proof}

\begin{proposition}\label{prop:simple-semisimple-monod} \cref{thm:monod} is true for all connected \emph{simple} Lie groups with finite centre if and only if it is true for all connected \emph{semisimple} Lie groups with finite centre.
\end{proposition}

\begin{proof}If $G$ is semisimple then we have $\fr{g} \cong \fr{g}_1 \oplus \cdots\oplus \fr{g}_k$ with each $\fr{g}_i$ simple. But $\rm{Int}(\fr{g}) \cong \rm{Int}(\fr{g}_1) \times \cdots \times \rm{Int}(\fr{g}_k)$ has the same Lie algebra, so applying \cref{prop:lie-algebra-monod} we may assume $G = \rm{Int}(\fr{g})$. There is a commutative diagram
\begin{center}
    \begin{tikzcd}
    \bigotimes_{i=1}^k H^*_{\rm{bd,cts}}(\rm{Int}(\fr{g}_i);\mathbb{R}) \rar \dar
    & H^*_{\rm{bd,cts}}(\prod_{i=1}^k \rm{Int}(\fr{g}_i);\mathbb{R}) \dar
         \\[-5pt]
   \bigotimes_{i=1}^k H^*_{\rm{cts}}(\rm{Int}(\fr{g}_i);\mathbb{R}) 
        \rar{\cong}
    & H^*_{\rm{cts}}(\prod_{i=1}^k \rm{Int}(\fr{g}_i);\mathbb{R}),
\end{tikzcd}
\end{center}
where the horizontal maps are induced by the cup product and pulling back along the projections, and the left vertical map is given by the comparison map from bounded continuous to continuous cohomology on each factor. The bottom horizontal map is an isomorphism by the K\"unneth theorem for continuous cohomology \cite[Theorem IX.6.10]{BorelWallach} (this may also be deduced from the van Est isomorphism). This reduces the surjectivity of the comparison map for $G$ to that for the simple factors $G_i$.
\end{proof}

\subsection{Classification of simple Lie groups}\label{sec:classification}  Combining \cref{prop:lie-algebra-monod} and \cref{prop:simple-semisimple-monod}, we see that it suffices to prove \cref{thm:monod} for all simple Lie algebras $\fr{g}$ in the following sense: we must prove it for a single connected Lie group $G$ with finite centre and Lie algebra $\fr{g}$. This group $G$ can be the adjoint group but can also be an alternative choice if more convenient. Simple real Lie algebras come, up to isomorphism, in three forms \cite{Knapp1996}:
\begin{enumerate}[\!\!\!(A)]
    \item Compact real forms of complex simple Lie algebras \cite[Table I]{helgason2001differential} \cite[Table 1]{VGO}.
    \item Complex simple Lie algebras viewed as real Lie algebras \cite[Table I]{helgason2001differential} \cite[Table 1]{VGO}.
    \item Noncompact real forms of complex simple Lie algebras \cite[Table II]{helgason2001differential} \cite[Table 4]{VGO}.
\end{enumerate}
Case (A) is in fact trivial: \cref{thm:monod} is true because the associated adjoint group will be a compact Lie group and the van Est isomorphism implies $H^*_\rm{cts}(G;\bb{R})$ is trivial. In this paper we will first prove \cref{thm:monod} in case (B) and then reduce case (C) to case (B). 

\section{The proof of \cref{thm:monod} for $\SL_n(\bb{C})$} In the section we prove \cref{thm:monod} for $\SL_n(\bb{C})$, by constructing cocycles for the Borel classes and then merely observing these have the desired properties. This is a special case of case (B) in \cref{sec:classification}, in fact the crucial one.

\smallskip

The van Est isomorphism in guise of \cref{prop:van-est-complex} yields a free exterior algebra
\[H^*_\rm{cts}(\SL_n(\bb{C});\bb{R}) \cong H^*_\rm{sing}(\SL_n(\bb{C});\bb{R}) \cong \Lambda^* \langle \ol{c}_{3},\ol{c}_{5},\ldots,\ol{c}_{2n-1}\rangle \]
on generators $\ol{c}_{2i-1}$ that are the unique primitive elements that transgress to the Chern classes $c_i \in H^{2i}(\rm{B}\rm{SL}_n(\bb{C});\bb{R})$ modulo decomposables. (Since the Chern class $c_i$ has degree $2i$, there is a regrettable tension between the notation $\ol{c}_{2i-1}$ and $c_i$.) A similar computation yields $H^*_\rm{cts}(\GL_n(\bb{C});\bb{R})  \cong \Lambda^* \langle \ol{c}_{1},\ol{c}_{3},\ol{c}_{5},\ldots,\ol{c}_{2n-1}\rangle$ and restriction along the inclusion $\SL_n(\bb{C}) \to \GL_n(\bb{C})$ zero on $\ol{c}_{1}$ and the identity on the other generators. The Borel class $\rm{Bo}_{n,2i-1}$ for $i \geq 2$ is, up to a nonzero constant, given by the image of the $\ol{c}_{2i-1}$ under the forgetful map $H^*_\rm{cts}(\GL_n(\bb{C});\bb{R}) \to H^*(\GL_n(\bb{C});\bb{R})$ (recall the right side is the cohomology of $\GL_n(\bb{C})$ as a discrete group).

\begin{corollary}\label{cor:cocycle-bon-stabilised} The Borel class $\rm{Bo}_{m+n,2n-1} \in H^{2n-1}(\GL_{m+n}(\C);\R)$ is, up to a nonzero multiple, represented by an inhomogeneous cocycle satisfying for generic elements (with notation as in \cref{thm:extension-criterion})
    \begin{align*} p_\bb{R}(\sf{g}_n)^{(m)} \colon \GL_{m+n}(\C)^{2n-1} &\lra \bb{R} \\
    [g_1|\compactcdots|g_{2n-1}] &\longmapsto (-1)^{m} \quad \sum_{\mathclap{\substack{
i_0+\compactcdots+i_{2n-1}=m\\
i_0,\compactldots, i_{2n-1}\ge 0
}}} \quad p_\bb{R}(\sf{g}_n)\Big((\!( V_{i_0,\ldots,i_{2n-1}}  \mid v_0^{i_0+1},\compactldots, v_{2n-1}^{i_{2n-1}+1})\!)\Big).\end{align*}
\end{corollary}

\begin{proof}
    In \cite[Theorems D, E]{KRS2} we constructed an inhomogeneous cocycle which represents, up to a nonzero multiple, the Borel class $\rm{Bo}_{n,2n-1} \in H^{2n-1}(\GL_n(\bb{C});\bb{R})$. It is given on generic elements by the map
   \begin{align*} \GL_n(\C)^{2n-1} &\lra \bb{R} \\
    [g_1|\compactcdots|g_{2n-1}] &\longmapsto \rm{GLi}_n(\ell_0,\compactldots,\ell_{n-1},\ell_{2n-1},\compactldots,\ell_n),\end{align*}
where $\ell_i \coloneq g_1 \compactcdots g_i \cdot \ell$ for a fixed line $\ell \in \bb{P}^1(\C^n)$ and $\rm{GLi}_n$ is the Grassmannian cluster polylogarithm of \cite{MR22} (one must make a choice of normalisation of the real period map, but it does not matter for our argument which one chooses). In the notation of \cref{sec:extension}, we recognise that this can be written as the composition of $F_0 \colon C_{2n-1}(\GL_n(\C)^{2n-1}) \to C_{2n-1}(n)$ with the map
\begin{align*}p_\bb{R}(\sf{g}_n) \colon C_{2n-1}(n) &\lra \bb{R} \\
(\!(v_0,\compactldots,v_{2n-1})\!) &\longmapsto \rm{GLi}_n(\ell_0,\compactldots,\ell_{n-1},\ell_{2n-1},\compactldots,\ell_n)\end{align*}
where $\ell_i$ is the span of $v_i$. This puts us in the context of \cref{thm:extension-criterion}, as by applying $p_\bb{R}$ to \cite[Corollary 6.15, 6.23]{KRS1} we get that $p_\bb{R}(\sf{g}_n)$ is a bi-Grassmannian cocycle. Thus these classes extend to $\GL_{m+n}(\bb{C})$ by the given formula outside of subset of $\GL_{n+m}(\bb{C})^{2n-1}$ contained in a finite union of codimension one submanifolds. By \cite[Section 6.4.1]{KRS2}, the Grassmannian cluster polylogarithm can be expressed as a finite sum of multiple polylogarithms and hence as a finite sum of Hodge correlators as in \cite[Section 7.1.1]{KRS2}. These are continuous by \cite[Proposition 7.3]{KRS2} (which in turn uses \cite{Mal20}) and using that induced map on continuous cohomology for $\GL_n(\bb{C}) \to \GL_{m+n}(\bb{C})$ is injective in degrees $\leq 2n-1$, they must be represent the Borel classes.
\end{proof}

We now now prove \cref{thm:monod} for $\SL_n(\bb{C})$:

\begin{proof}[Proof of \cref{thm:monod} for $\SL_n(\bb{C})$]
We must prove that $H^*_\rm{bd,cts}(\SL_n(\C);\R) \to H^*_\rm{cts}(\SL_n(\bb{C});\bb{R})$ is surjective, and by the above considerations it suffices to prove that each of the Borel classes $\rm{Bo}_{n,2i-1}$ for $2 \leq i \leq n$ is represented by a bounded continuous cocycle on $\GL_n(\bb{C})$. In \cref{sec:continuous-cohomology-etc} we saw that to define continuous cohomology we may use measurable functions. Similarly, \cite[Proposition 2.4]{MonodSLn} says that to define bounded continuous cohomology we may use essentially bounded measurable functions. (By its proof, these produce bounded continuous cocycles by a mollification argument as in \cite[\S 4]{Blanc}.) \cref{cor:cocycle-bon-stabilised} gives a cocycle representative for $\rm{Bo}_{n,2i-1}$ and these are essentially bounded continuous cocycles as Hodge correlators are bounded continuous functions for $n \geq 2$ by \cite[Proposition 7.3]{KRS2}.
\end{proof}

\begin{remark}The Borel class $\rm{Bo}_{n,1} \in H^1(\GL_n(\bb{C});\bb{R})$ does \emph{not} admit a bounded continuous representative. It is, up to a nonzero multiple, represented by $[g] \mapsto \log|{\det(g)}|$ \cite{Hamida}.\end{remark}

\section{The proof of \cref{thm:monod} for complex simple Lie groups} In this section we prove \cref{thm:monod} for the complex simple Lie algebras (considered as real simple Lie algebras) that appear as case (B) in \cref{sec:classification}. The classification of complex simple Lie algebras has four infinite families and five exceptional cases \cite[Table I, p.~346]{helgason2001differential}: A, B, C, D, and E6, E7, E8, F4, G2. To do so, we first need a more convenient description of the continuous cohomology of a connected complex semisimple Lie group with finite centre.

\subsection{The van Est isomorphism for complex simple Lie groups} Our goal in this subsection is to prove the following computational tool:

\begin{proposition}\label{prop:van-est-complex} Let $G$ be a connected complex semisimple Lie group with finite centre.
    \begin{enumerate}[(i)]
    \item \label{enum:van-est-complex-i} There is an isomorphism $H^*_\rm{cts}(G;\bb{R}) \cong H^*_\rm{sing}(G;\bb{R})$.
    \item \label{enum:van-est-complex-ii} Given a homomorphism $f \colon G \to G'$ of connected complex semisimple Lie groups with finite centre, we can choose these isomorphisms so that the following diagram commutes
    \[\begin{tikzcd} H^*_\rm{cts}(G;\bb{R}) \rar{\cong} & H^*_\rm{sing}(G;\bb{R}) \\[-5pt]
     H^*_\rm{cts}(G';\bb{R}) \rar{\cong} \uar{f^*} & H^*_\rm{sing}(G';\bb{R}) \uar{f^*}.\end{tikzcd}\]
    \end{enumerate}
\end{proposition}

\begin{proof}As usual, let $\fr{k} \subseteq \fr{g}$ denote the Lie algebras of $K \subseteq G$. Cartan decompositions of a complex semisimple Lie algebra $\fr{g}$ are in bijection with compact real forms $\fr{u} \subseteq \fr{g}$, where $\fr{k} = \fr{u}$ and $\fr{p}$ is given by $i \fr{u}$. The Lie algebra $\fr{g}_\rm{u}$ is isomorphic to $\fr{u} \oplus i\fr{u}$ as a vector space with Lie bracket given by $[-,-]_\rm{u}$. The map
\begin{align*} \rho \colon \fr{g}_\rm{u} &\lra \fr{u} \oplus \fr{u} \\
(x,iu) & \longmapsto (x+u,x-u)\end{align*}
is then an isomorphism of Lie algebras, if on the target we use the direct sum Lie bracket:
\begin{align*}\rho([(u,iv),(u',iv')]_\rm{u} &= \rho([u,u']+[v,v'],i[u,v']+i[v,u']) \\& = ([u,u']+[v,v']+[u,v']+[v,u'],[u,u']+[v,v']-[u,v']-[v,u']) \\
&= ([u+v,u'+v'],[u-v,u'-v']) \\
&= [(u+v,u-v),(u'+v',u'-v')] \\
&= [\rho(u,iv),\rho(u',iv')].\end{align*}
Under this isomorphism, $\fr{k} = \fr{u}$ is mapped to the diagonal $\fr{u} = \{(x,x) \mid x \in \fr{u}\} \subseteq \fr{u} \oplus \fr{u}$. Taking adjoint groups, this identifies $G_\rm{u} = \rm{Int}(\fr{g}_\rm{u})$ with $\rm{Int}(\fr{u}) \times \rm{Int}(\fr{u})$ and its maximal compact $K$ corresponding to $\fr{k} = \fr{u} \subseteq \fr{g}$ is given by the diagonal copy $\rm{Int}(\fr{u})$ and the compact dual symmetric space can be identified as $(\rm{Int}(\fr{u}) \times \rm{Int}(\fr{u}))/\rm{Int}(\fr{u}) \cong \rm{Int}(\fr{u})$ using $(u,v) \mapsto uv^{-1}$. So we have isomorphisms
\[H^*_\rm{cts}(G;\bb{R}) \overset{\cong}{\longleftarrow} H^*(\fr{g},\fr{u}) \overset{\cong}\lra H^*(\fr{g}_\rm{u},\fr{u}) \overset{\cong}\lra H^*(\fr{u} \oplus \fr{u},\fr{u}) \overset{\cong}\lra H^*_\rm{sing}(\rm{Int}(\fr{u});\bb{R}).\] 
Since $K \to \rm{Int}(\fr{u})$ is a finite central cover and $K \subseteq G$ is a homotopy equivalence, it remains to observe there are isomorphisms of cohomology groups of spaces 
\[H^*_\rm{sing}(\rm{Int}(\fr{u});\bb{R}) \cong H^*_\rm{sing}(K;\bb{R}) \cong H^*_\rm{sing}(G;\bb{R}).\]
This completes the proof of \eqref{enum:van-est-complex-i}.

For \eqref{enum:van-est-complex-ii}, observe that the isomorphisms are determined completely by the choice of compact real form $\fr{u} \subseteq \fr{g}$: tracing through the isomorphisms, the diagram commutes if we can choose $\fr{u} \subseteq \fr{g}$ and $\fr{u}' \subseteq \fr{g}'$ so that $f_*(\fr{u}) \subseteq \fr{u}'$. Such a choice can be made by \cref{lem:compatible-real-forms}.
\end{proof}

\begin{lemma}\label{lem:compatible-real-forms} Let $h \colon \fr{g} \to \fr{g}'$ be a map of semisimple complex Lie algebras. For any compact real form $\fr{u} \subseteq \fr{g}$ there exists a compact real form $\fr{u}' \subseteq \fr{g'}$ such that $h(\fr{u}) \subseteq \fr{u}'$.
\end{lemma}

 \begin{proof}
    We can factor $h$ as $\fr{g} \twoheadrightarrow \fr{g}/\ker(\rho) \hookrightarrow \fr{g}'$, so it suffices to prove the result for each of these maps separately, and hence we have two cases to consider: 

    \smallskip

    \noindent \emph{Case 1:} The map $h$ is surjective so $\fr{g}'=\fr{g}/\ker(h)$. By definition $\ker(h) \subset \fr{g}$ is an ideal. Since $\fr{g}$ is semisimple, we can write $\fr{g}=\bigoplus_{i \in I} \fr{g}_i$ as a sum of simple Lie algebras and must have $\ker(\rho)=\bigoplus_{j \in J} \fr{g}_j$ for some $J \subseteq I$ \cite[II.6.3]{helgason2001differential}. Thus, $\fr{g}/\ker(\rho)= \bigoplus_{k \in I \setminus J} \fr{g}_k$ and there is a canonical Lie algebra isomorphism $\fr{g} \cong \ker(h) \oplus \fr{g}/\ker(h)$. Let $\fr{u} \subseteq \fr{g}$ be an arbitrary compact real form. For compact real forms $\fr{v} \subseteq \ker(h)$ and $\fr{w} \subseteq \fr{g}/\ker(h)$, $\fr{v} \oplus \fr{w} \subseteq \fr{g}$ is also a compact real form. 
    By \cite[4.1.2]{VGO} there exists $f \in \rm{Int}(\fr{g})$ such that $f(\fr{v} \oplus \fr{w})=\fr{u}$. As $\rm{Int}(\fr{g})=\rm{Int}(\ker(\rho)) \times \rm{Int}(\fr{g}/\ker(\rho))$, we have $f=(f_1,f_2)$ for some $f_1 \in \rm{Int}(\ker(h))$ and $f_2 \in \rm{Int}(\fr{g}/\ker(h))$. Then taking $\fr{u}'=f_2(\fr{w})$ gives the required compact form.

    \smallskip

    \noindent \emph{Case 2:} The map $h$ is injective so $\fr{g} \subseteq \fr{g}'$. Pick an arbitrary compact real form $\fr{u} \subset \fr{g}$. Since $\fr{g}$ is semisimple, $\fr{u}$ is also semisimple. Let $G_u \subset \rm{Int}(\fr{g}')$ be the connected Lie subgroup corresponding to $\fr{u} \subseteq \fr{g} \subseteq \fr{g}'$. Since $\fr{u}$ is compact semisimple the group $G_u$ must be compact. Hence, there exists a maximal compact subgroup $K \subset \rm{Int}(\fr{g}')$ such that $G_u \subset L$ and hence $\fr{u}$ is contained in the Lie algebra $\fr{k}$ of $K$. Taking $\fr{u'} = \fr{k}$ gives the required compact form.\end{proof}

\begin{remark}One may in fact verify that the isomorphisms in \cref{prop:van-est-complex} \eqref{enum:van-est-complex-i} are independent of the choice of the compact real form $\fr{u}$. We do not need this.\end{remark} 

\subsection{Lie algebra cohomology and Casimir elements} \label{sec:casimir} For a connected complex semisimple Lie group $G$ with finite centre, both $H^*(G;\bb{C})$ and $H^*(\rm{B}G;\bb{C})$ have concise algebraic descriptions in terms of its Lie algebra $\fr{g}$.

Firstly, the Chevalley--Eilenberg cochain complex $C^*(\fr{g})=(\Lambda^*(\fr{g}^\vee),d)$ with $\fr{g}^\vee$ placed in degree 1, can be identified with the cochain complex $\Omega^*(G)^G$ of left $G$-invariant $\bb{C}$-valued differential forms on $G$. The inclusion $\Omega^*(G)^G \hookrightarrow \Omega^*(G)$ induces an isomorphism $H^*(\fr{g}) \cong H^*_\rm{sing}(G;\bb{C})$ \cite[Section 15, 17]{chevalley1948cohomology}. Viewing $\fr{g}$ as a $\fr{g}$-module via the adjoint representation and $\fr{g}^\vee$ as the dual representation, we can consider $\Lambda^*(\fr{g}^\vee)$ as a graded $\fr{g}$-representation. Any $\fr{g}$-invariant element $\omega \in (\Lambda^*(\fr{g}^\vee))^\fr{g}$ is closed with respect to the Chevalley--Eilenberg differential, and sending an invariant element to its cohomology class gives an isomorphism of algebras \cite[Theorem 19.1]{chevalley1948cohomology}
\[(\Lambda^*(\fr{g}^\vee))^\fr{g} \overset{\cong}\lra H^*(\fr{g}) \overset{\cong}\lra H^*_\rm{sing}(G;\bb{C}).\]
The domain is a finite-dimensional Hopf algebra, so must be an exterior algebra $\Lambda^* \langle x_1,\ldots,x_r \rangle$ on primitive generators in odd degree.

Secondly, the Chern--Weil construction gives an isomorphism of algebras
\[S^*(\fr{g}^\vee)^\fr{g} \overset{\cong}\lra H^*(\rm{B}G;\bb{C})\] 
where on the left side $\fr{g}^\vee$ is placed in degree 2. It is rendered more computable using the Chevalley restriction theorem, which says that there is a restriction map $S^*(\fr{g}^\vee)^\fr{g} \to S^*(\fr{h}^\vee)^W$ for $\fr{h} \subseteq \fr{g}$ a Cartan subalgebra and $W$ the Weyl group, and that this is an isomorphism \cite[Section 23.1]{Humphreys}. By the Chevalley--Shephard--Todd theorem this is a polynomial algebra on generators $\bb{C}[y_1,\ldots,y_s]$ in even degree \cite{Chevalley}.

\smallskip

To connect these we use the fibre sequence $G \to \ast \to \rm{B}G$, or more precisely, the associated transgression for the cohomological Serre spectral sequence $E^{p,q}_r = H^*(\rm{B}G;\bb{C}) \otimes H^*(G;\bb{C}) \Rightarrow H^{p+q}(\ast;\bb{C})$ given by
\[\tau \colon E_{p-1}^{0,p-1}  \lra E_{p-1}^{p,0}\]
where $E_{p-1}^{0,p-1}$ is the subspace of $H^{p-1}(G;\bb{C})$ of transgressive elements and $E_{p-1}^{p,0}$ is a quotient of $H^p(\rm{B}G;\bb{C})$ (see \cite[Section 1]{Carlson} for a recent survey). We have \cite{Cartan}:

\begin{theorem}[Cartan] \label{thm:cartan} \,
\begin{enumerate}[(i)]
    \item There are equalities $E_{p-1}^{0,p-1} = \prim\,H^{p-1}(G;\bb{C})$ and $E_{p-1}^{p,0} = \rm{indec}\, H^p(\rm{B}G;\bb{C})$.
    \item The transgression induces an isomorphism
\[\tau \colon E_{p-1}^{0,p-1} = \prim\,H^{p-1}(G;\bb{C}) \lra E_{p-1}^{p,0} = \rm{indec}\, H^p(\rm{B}G;\bb{C}).\]
\end{enumerate}
\end{theorem}

In particular, we have $r=s$ and may choose isomorphism $H^*(G;\bb{C}) \cong \Lambda^* \langle x_1,\ldots,x_r \rangle$ and $H^*(\rm{B}G;\bb{C}) \cong \bb{Q}[y_1,\ldots,y_r]$ with $|x_i| = 2j_i-1$, $|y_i| = 2j_i$, $x_i$ primitive, and $\tau(x_i) = y_i$ modulo decomposables.

\smallskip

We will also need a different algebraic description of $H^*(\rm{B}G;\bb{C})$, in order to extract results from the literature. We use the Killing form to $\fr{g}$-equivariantly identify $\fr{g}$ with $\fr{g}^\vee$. The Poincar\'e--Birkhoff--Witt theorem says that the inclusion of $\fr{g}$ into its the (complex) universal enveloping algebra $U(\fr{g})$, yields a natural isomorphism of graded $\fr{g}$-modules
\[\lambda \colon S^*(\fr{g}) \overset{\cong}\lra \gr^*\,U(\fr{g}),\] 
where the filtration on $U(\fr{g})$ is induced by the filtration $T(\fr{g})$ obtained by placing $\fr{g}$ in filtration degree 1. Moreover, $\lambda$ maps the subalgebra $(S^*(\fr{g}))^\fr{g}$ of invariants into the centre $\gr^*\,Z(U(\fr{g}))$. It follows from the proofs in \cite[p.~94]{Gauger} \cite[Lemma 1]{Berdjis} that elements $x_1,\ldots,x_k \in Z(U(\fr{g}))$ of filtration degree $m_1,\ldots,m_k$ generate $Z(U(\fr{g}))$ if and only if $\lambda^{-1}(\ol{x}_1),\ldots,\lambda^{-1}(\ol{x}_k)$ generate $(S^*(\fr{g}))^\fr{g}$, where $\ol{x}_i$ is the image in $\gr^{m_i} Z(U(\fr{g}))$. For $m \in U(\fr{g})$ of specified filtration degree $m$, we will write $\sigma_m(x)$ for the image of $\lambda^{-1}(\ol{x})$ in $S^m(\fr{g}^\vee)$ under the identification using the Killing form, thinking of it as a ``symbol''.

\subsection{The proof of \cref{thm:monod} for complex simple Lie algebras} We first deal with the complex simple Lie algebras $\fr{g}$ that are \emph{not} in the infinite family D. 

\begin{proof}[Proof of \cref{thm:monod} for complex simple Lie groups not in the infinite family D] We fix a non-trivial complex representation of smallest dimension $\rho \colon \fr{g} \to \fr{gl}_n$, and is induced by a homomorphism $G \to \GL_n(\bb{C})$ for some connected complex Lie group $G$ with finite centre and Lie algebra $\fr{g}$. Since $G$ is semisimple this factors over $\SL_n(\bb{C})$ (the composition $G \to \GL_n(\bb{C}) \to \bb{C}^*$ with the determinant is trivial because its image is semisimple). Using naturality and the fact that \cref{thm:monod} is known for $\SL_n(\bb{C})$, it suffices to prove the following map is surjective
\[H^*_\rm{cts}(\SL_n(\bb{C});\bb{R}) \lra H^*_\rm{cts}(G;\bb{R}).\]

By \cref{prop:van-est-complex} it suffices to prove that 
\[H^*_\rm{sing}(\SL_n(\bb{C});\bb{C}) \lra H^*_\rm{sing}(G;\bb{C}).\]
is surjective. We claim that if $\fr{g}$ is a complex simple Lie algebra not in the family $D$, then \cite[Corollary 2.4]{Casimirelements} (see its references for proofs, or \cite{Molev}) implies the restriction map
\[\rho^\vee \colon (S^*(\fr{gl}_n^\vee))^{\fr{gl}_n} \lra (S^*(\fr{g}^\vee))^\fr{g}\] 
induced by $\rho$ is surjective. Assuming this, by the naturality of the Chern--Weil construction the map $H^*(\rm{B}\GL_n(\bb{C});\bb{C}) \to H^*(\rm{B}G;\bb{C})$ is surjective.
Since $G \to \GL_n(\C)$ factors trough $\SL_n(\C)$ then $H^*(\rm{B}\SL_n(\bb{C});\bb{C}) \to H^*(\rm{B}G;\bb{C})$ is also surjective.  
By \cref{thm:cartan} and naturality of the transgression, the map $H^*_\rm{sing}(\SL_n(\bb{C});\bb{R}) \to H^*_\rm{sing}(G;\bb{C})$ is then also surjective as required.

\smallskip

We now establish the claim, which requires unwinding of notation. First, let $\langle-,-\rangle$ denote the (normalised) Killing form, pick a basis $J_1,\ldots,J_d$ of $\fr{g}$, and let $J^1,\ldots,J^d$ be the basis of $\fr{g}$ satisfying $\langle J_i,J^j \rangle = \delta_{ij}$. The target of the identification $\fr{gl}_n \cong \rm{End}(\bb{C}^n)$ is an associative algebra, so $\rho$ extends to a map $\widetilde{\rho} \colon U(\fr{g}) \to \rm{End}(\bb{C}^n)$ of associative algebras. Now set 
\[\ol{\Omega} \coloneq \sum_{i=1}^d J_i \otimes J^i \in U(\fr{g}) \otimes U(\fr{g}) \quad \text{and} \quad  
F \coloneq -1(\id \otimes \widetilde{\rho})(\ol{\Omega}) = -\sum_{i=1}^d J_i \otimes \rho(J^i) \in U(\fr{g}) \otimes \rm{End}(\bb{C}^n).\]
Interpreting $U(\fr{g}) \otimes \rm{End}(\bb{C}^n)$ as $(n \times n)$-matrices with entries in $U(\fr{g})$, we see it is an associative algebra with trace $\rm{tr} \colon U(\fr{g}) \otimes \rm{End}(\bb{C}^n) \to U(\fr{g})$. The statement of \cite[Corollary 2.4]{Casimirelements} is that the elements $\rm{Tr}(F^m)$ for $m \geq 2$ generate $Z(U(\fr{g}))$ (in fact, only some are needed \cite[Table 1]{Casimirelements}).

By the final paragraph of \cref{sec:casimir} and noting that $\rm{Tr}(F^m) \in U(\fr{g})$ lies the filtration degree $m$, it suffices to prove its symbol $\sigma_m(\rm{Tr}(F^m)) \in S^m(\fr{g}^\vee)^\fr{g}$ is in the image of the map $\rho^\vee$. Write $\rho(J^i) = \sum_{1\leq k,\ell \leq n} \rho(J^i)_{k\ell} E_{k\ell}$ with $E_{k\ell} \in \rm{End}(\bb{C}^n)$ the matrix with unique nonzero entry given by $1$ in the $(k,\ell)$ entry and $\rho(J^i)_{k\ell} \in \bb{C}$, so that
\[F = -\sum_{1 \leq k,\ell \leq n}\left(\sum_{i=1}^d J_i \rho(J^i)_{k\ell}\right) \otimes E_{k\ell} \in \fr{g} \otimes \rm{End}(\bb{C}) \subset U(\fr{g}) \otimes \rm{End}(\bb{C}).\]
Under the identification $\fr{g} \cong \fr{g}^\vee$ using the Killing form, $J_i$ is sent to $\langle J_i,-\rangle \in \fr{g}^\vee$. Because the functional $X \mapsto \smash{\sum_{i=1}^d \langle J_i,X\rangle \, \rho(J^i)_{k\ell}}$ sends $J^j$ to $\rho(J^j)_{k\ell}$, it is equal to the functional $X \mapsto \rho(X)_{k\ell}$ extracting the $(k,\ell)$ matrix entry of the representation. In other words, $(\sigma_1 \otimes \id)(F) = -(\rho^\vee \otimes \id)(E)$ with $E = \sum_{1 \leq k,\ell \leq n} e_{k,\ell} \otimes E_{k,\ell} \in \fr{gl}_n^\vee \otimes \rm{End}(\bb{C}^n)$, where $e_{k,\ell} \in \fr{gl}_n^\vee$ is the linear functional extracting the $(k,\ell)$ matrix entry. We now compute
\begin{align*}\sigma_m(\rm{tr}(F^m)) &= \rm{tr}((\sigma_m \otimes \id)(F)^m) \\
&= (-1)^m\rm{tr}((\rho^\vee \otimes \id)(E)^m) \\
&= (-1)^m \rho^\vee(\rm{tr}(E^m)).\end{align*}
The first equation uses that the symbol map $\sigma_m$ satisfies (i) $\sigma_m(x+y) = \sigma_m(x)+\sigma_m(y)$ if $x,y$ have filtration degree $m$ and (ii) $\sigma_{m_1+m_2}(xy) = \sigma_{m_1}(x)\sigma_{m_2}(y)$ if $x$ and $y$ have filtration degrees $m_1$ and $m_2$. The second equation uses $(\sigma_1 \otimes \id)(F) = -(\rho^\vee \otimes \id)(E)$. The third equation uses that $\rho^\vee \otimes \id$ is an algebra homomorphism and that $\rm{tr}$ intertwines $\rho^\vee$ and $\rho^\vee \otimes \id$. This completes the proof of the claim. 
\end{proof}

It remains to deal with the remaining case of the infinite family D. 

\begin{proof}[Proof of \cref{thm:monod} for complex simple Lie groups in the infinite family D] If we follow the strategy of the previous case, \cite[Corollary 2.4]{Casimirelements} explains what remains to be done: we need to prove that the Pfaffian is pulled back from some representation. This was proven in \cite[Section 2.7]{Gon99}: we can take the complex spinor representation.
\end{proof}

\section{The proof of \cref{thm:monod} for real simple Lie groups} We now deal with case (C) in \cref{sec:classification}: in the classification of noncompact real forms there are seven infinite families and twelve exceptional cases \cite[Table II]{helgason2001differential} \cite[Table 4]{VGO}. The work of Harnick and Ott deals with most cases, and the remaining ones are handled by relating the map induced by the inclusion $G \to G_\bb{C}$ on continuous cohomology to the Cartan embedding of compact dual symmetric spaces.

\subsection{The work of Harnick and Ott} Recall from \cref{sec:compact-duals} that to any connected noncompact real Lie group $G$ with maximal compact subgroup $K$, we associate a dual compact Lie group $G_\rm{u}$ and maximal compact subgroup $K_\rm{u}$. Hartnick and Ott proved that \cref{thm:monod} is true if $\rm{rk}(G_u) = \rm{rk}(K)$ \cite[Theorem 2, Proposition 2.3]{HartnickOtt}. 

They observe this hypothesis is satisfied for the hermitian Lie groups \cite[Theorem 1]{HartnickOtt}: in the notation of \cite[Table II, p.~354]{helgason2001differential}, there are the four infinite families AIII, DIII, BDI (for $q=2$) and CI, and the two exceptional cases EIII and EVII \cite[Theorem 6.1, Table II, p.~354]{helgason2001differential}. However, it also satisfied by two further infinite families CII and BDI (when not both $p$, $q$ are odd). Let us explain this: BDI has $G_\rm{u} = \rm{SO}_{p+q}$ and $K = \rm{SO}_p \times \rm{SO}_q$, with $\rm{rk}(\rm{SO}_n = \lfloor \tfrac{n}{2} \rfloor$, and CII has $G_\rm{u} = \rm{Sp}_{p+q}$ and $K = \rm{Sp}_p \times \rm{Sp}_q$, with $\rm{rk}(\rm{Sp}_n) = n$. It is also satisfied by nearly all of remaining exceptional cases: by \cite[Table II, p.~354]{helgason2001differential}, one has $\rm{rk}(G_\rm{u}) \neq \rm{rk}(K)$ only for the cases EI and EIV.

\smallskip

To summarise: it remains for us to deal in this section with the three infinite families AI, AII, BDI (for $p$, $q$ both odd) and the two remaining exceptional cases EI and EIV.

\subsection{The Cartan embedding} We now establish the main technical result needed to establish the remaining cases. Take a real simple Lie algebra $\fr{g}$ arising from a Cartan decomposition of a complex Lie algebra $\fr{g}_\bb{C}$, i.e.~a real form of $\fr{g}_\bb{C}$. Then there is an inclusion $\fr{g} \to \fr{g}_{\bb{C}}$, and suppose that we have a map $G \to G_\bb{C}$ of Lie groups inducing the given map on Lie algebras. Fixing maximal compact subgroup $K \subseteq G$ and letting $G_\rm{u} = \rm{Int}(\fr{g}_u)$ be the compact dual of $G$ with corresponding maximal compact $K_\rm{u} \subseteq G_\rm{u}$, there is a Cartan embedding
\begin{align*}\iota_\theta \colon G_\rm{u}/K_\rm{u} &\lra G_\rm{u} \\
gK &\longmapsto g\Theta_\rm{u}(g)^{-1}.\end{align*}
Here $\Theta_\rm{u}$ is the automorphism of $G_\rm{u} = \rm{Int}(\fr{g}_\rm{u})$ induced by the automorphism of $\fr{g}_\rm{u} = \fr{k} \oplus \fr{p}$ given by $\id_\fr{k} \oplus -\id_\fr{p}$.

\begin{proposition}\label{prop:cartan-embedding} There is a commutative square
\[\begin{tikzcd} H^*_\rm{cts}(G_\bb{C};\bb{R}) \rar{\rm{vE}_{G_\bb{C}}} \dar & H^*_\rm{sing}(G_\rm{u};\bb{R}) \dar{\iota_\theta^*} \\[-5pt]
H^*_\rm{cts}(G;\bb{R}) \rar{\rm{vE}_{G}} & H^*_\rm{sing}(G_\rm{u}/K_\rm{u};\bb{R}) \end{tikzcd}\]
with right vertical map induced by the Cartan embedding.
\end{proposition}

\begin{proof}We may assume that $G = \rm{Int}(\fr{g})$ and $G_\bb{C} = \rm{Int}(\fr{g}_\bb{C})$. A compact real form $\fr{u} \subseteq \fr{g}_\bb{C}$ determines a Cartan decomposition of $\fr{g}$ as $\fr{k} = \fr{g} \cap \fr{u}$ and $\fr{p} = \fr{g} \cap i \fr{u}$ \cite[III.7.4]{helgason2001differential} (note that with this choice, $\fr{g}_\rm{u}$ has a preferred identification with $\fr{u}$). Doing so, the left square of the below diagram commutes by naturality of the van Est isomorphism
\[\begin{tikzcd} H^*_\rm{cts}(G_\bb{C};\bb{R}) \rar{\rm{vE}_{G_\bb{C}}}[swap]{\cong} \dar & H^*(\fr{g}_\bb{C},\fr{u}) \rar[swap]{\cong} \dar & H^*_\rm{sing}(G_\rm{u};\bb{R}) \dar \\[-5pt]
H^*_\rm{cts}(G;\bb{R}) \rar{\rm{vE}_{G}}[swap]{\cong} & H^*(\fr{g},\fr{k}) \rar[swap]{\cong} & H^*_\rm{sing}(G_\rm{u}/K_\rm{u};\bb{R}).\end{tikzcd}\]
To see that the right square involving \cref{prop:van-est-symmetric} commutes, we expand these isomorphisms to
\[\begin{tikzcd} H^*(\fr{g}_\bb{C},\fr{u}) \rar{\upsilon_{\fr{g}_\bb{C},\fr{g}}}[swap]{\cong} \dar & H^*(\fr{g}_{\bb{C},\rm{u}},\fr{u}) \dar & H^*(\fr{u} \oplus \fr{u},\fr{u}) \rar{\mu_{\fr{u}\oplus \fr{u},\fr{u}}}[swap]{\cong} \lar{\cong}[swap]{\rho^*} \dar & H^*_\rm{sing}(G_\rm{u};\bb{R}) \dar{\iota_\theta^*} \\[-5pt]
H^*(\fr{g},\fr{k}) \rar{\upsilon_{\fr{g},\fr{k}}}[swap]{\cong} & H^*(\fr{g}_\rm{u},\fr{k}) \rar[equal] & H^*(\fr{g}_\rm{u},\fr{k}) \rar{\mu_{\fr{g},\fr{k}}}[swap]{\cong} & H^*_\rm{sing}(G_\rm{u}/K_\rm{u};\bb{R}).\end{tikzcd}\]
Here the left square commutes by naturality of the construction of the compact dual Lie algebra. The middle isomorphism on top is induced by the isomorphism $\rho$ from the proof of \cref{prop:van-est-complex} and the middle square commutes if we let $\fr{g}_u \to \fr{u} \oplus \fr{u}$ be given by $\fr{g}_u = \fr{k} \oplus \fr{p} \ni (k,p) \mapsto (k+p,k-p) \in \fr{u} \oplus \fr{u}$. This is induced by the homomorphism $G_\rm{u} \to G_\rm{u} \times G_\rm{u}$ given by $g \mapsto (g,\Theta_\rm{u}(g))$, which induces a map $G_\rm{u}/K_\rm{u} \to (G_\rm{u} \times G_\rm{u})/G_\rm{u}$. Finally, we can identify $(G_\rm{u} \times G_\rm{u})/G_\rm{u} \cong G_\rm{u}$ via $(u,v) \mapsto uv^{-1}$ in the top right corner, and hence the right square commutes if the rightmost vertical map is induced by the Cartan embedding.
\end{proof}

To use this, consider the composition
\[G_\rm{u} \lra G_\rm{u}/K \lra G_\rm{u}\]
sending $g$ to $g\Theta_\rm{u}(g)^{-1}$. On a primitive elements $x$ in $H^*(G_\rm{u};\bb{R})$ this is given by $x \mapsto x-\Theta_\rm{u}^*(x)$. Since $\Theta_\rm{u}$ is an involution, the rightmost term is given by $\Theta_\rm{u}^*(x) = \pm x$ so we must have $\iota^*_\theta(x) = 0$ or $2x$. This makes it well-suited for computations, since we can often determine the action of $\Theta_\rm{u}$ on cohomology.

\subsection{The proof of \cref{thm:monod} for the remaining real simple Lie algebras} We will now use the results of the previous subsection to deal with the remaining three infinite families AI, AII, BDI (for $p$, $q$ both odd) and remaining two exceptional cases EI and EIV.

\subsubsection{The infinite family AI} We first record some information about this family. This is the infinite family of real simple Lie algebras given by $\fr{sl}_n$. There is an associated family of noncompact connected real Lie groups with finite centre given by $\SL_n(\bb{R})$. The group $\SL_n(\bb{R})$ has a standard maximal compact subgroup $\rm{SO}_n \subseteq \SL_n(\bb{R})$ and a compact dual group $\rm{SU}_n$, with Cartan involution $A \mapsto \ol{A}$ \cite[p.~348]{helgason2001differential}. Thus the compact dual symmetric space is $\rm{SU}_n/\rm{SO}_n$, and  using \cref{prop:van-est-symmetric} we compute that
\[H^*_\rm{cts}(\SL_n(\bb{R});\bb{R}) \cong H^*_\rm{sing}(\rm{SU}_n/\rm{SO}_n;\bb{R}) \cong \begin{cases} \Lambda^* \langle \ol{c}_{5},\ldots,\ol{c}_{4m-3} \rangle & \text{if $n=2m-1$,}\\
\Lambda^* \langle \ol{c}_{5},\ldots,\ol{c}_{4m-3},e_{2m} \rangle & \text{if $n=2m$,}\end{cases}\] 
Here the class $e_{2n}$ in degree $2n$ is the pullback of the Euler class $e \in H^{2n}_\rm{sing}(\rm{BSO}_n;\bb{R})$ along the map $\rm{SU}_n/\rm{SO}_n \to \rm{BSO}_n$ and the classes $\ol{c}_{4i-3}$ in degree $4i-3$ are the unique elements whose pullback to $\rm{SU}_n$ along the quotient map are the unique primitive elements whose transgression is the $(2i-1)$st Chern class $c_{2i-1} \in H^{4i-2}(\rm{BSU}_n;\bb{R})$ up to decomposables. (Since the Chern class $c_i$ has degree $2i$, there is a regrettable tension between the notation $\ol{c}_{4i-3}$ and $c_{2i-1}$.)

\begin{proof}[Proof of \cref{thm:monod} for the infinite family AI] We now proceed with the proof. Firstly, the continuous cohomology corresponding class $e_{2n}$ is in the image of the geometric map \eqref{eqn:geometric-map} so represented by a bounded continuous cocycle by \cite[Propositions 2.1, 2.2]{HartnickOtt} (see also \cite[Section 3.1.3]{Karlsson}); an explicit construction appears in \cite{IvanovTuraev}.

Using that we know \cref{thm:monod} is true for $\SL_n(\bb{C})$ and there is a commutative diagram
\[\begin{tikzcd} H^*_{\rm{bd,cts}}(\SL_n(\C);\bb{R}) \rar \dar & H^*_\rm{cts}(\SL_n(\C);\bb{R}) \dar \\[-5pt]
H^*_{\rm{bd,cts}}(\SL_n(\R);\bb{R}) \rar & H^*_\rm{cts}(\SL_n(\R);\bb{R}), \end{tikzcd}\]
it suffices to prove that the generators $\ol{c}_{4i-3}$ are in the image of the right vertical map (as the notation suggests). To do so, we apply \cref{prop:cartan-embedding} to the Cartan embedding given by
\begin{align*} \iota_\theta \colon \rm{SU}_n/\rm{SO}_n &\lra \rm{SU}_n \\
A &\longmapsto A\ol{A}^{-1}\end{align*}
and note it suffices to compute the effect of the composition $\rm{SU}_n \to \rm{SU}_n/\rm{SO}_n \to \rm{SU}_n$ of the quotient map and the Cartan embedding on the generators $\ol{c}_{2i-1} \in H^*_\rm{sing}(\rm{SU}_n;\bb{R})$. These are primitive, so they get mapped to $\ol{c}_{2i-1} -\Theta_\rm{u}^*(\ol{c}_{2i-1})$ and it suffices to determine whether $\Theta_\rm{u}$ acts on $\ol{c}_{4i-3}$ by $+1$ or $-1$; in the former case the composition is nontrivial on $\ol{c}_{4i-3}$ and hence so is the Cartan embedding.

To determine this, we use the action of the involution $\Theta_\rm{u}$ on fibre sequence $\rm{SU}_n \to \ast \to \rm{BSU}_n$ and the fact that $\ol{c}_{2i-1}$ transgresses to the Chern class $c_i \in H^{2i}(\rm{BSU}_n;\bb{R})$ modulo decomposables. Complex conjugation $\Theta_\rm{u}$ extends to $\rm{BU}_n$ and as it acts by the first Chern class by $c_1 \mapsto -c_1$, the splitting principle tells us it acts on the higher Chern classes by $c_i \mapsto (-1)^i c_i$, and hence by naturality of transgression must act by $\ol{c}_{2i-1} \mapsto (-1)^i \ol{c}_{2i-1}$. In particular it acts by $-1$ on $\ol{c}_{4i-3}$.\end{proof}

\subsubsection{The infinite family AII} This is the infinite family of real simple Lie algebras given by $\fr{su}^*(2n)$, with associated family of noncompact connected real Lie groups with finite centre given by
\[\rm{SU}^*(2n) \coloneq \{A \in \SL_{2n}(\bb{C}) \mid AJ = J \ol{A}\} \qquad \text{with $J = \begin{bsmallmatrix} 0 & \id_n \\ -\id_n & 0 \end{bsmallmatrix}$.}\] 
It has a standard maximal compact subgroup $\rm{Sp}_n$ and compact dual group $\rm{SU}_{2n}$ with Cartan involution given by $A \mapsto J\ol{A}J^{-1}$ \cite[p.~348]{helgason2001differential}. Thus the compact dual symmetric space is $\rm{SU}_{2n}/\rm{Sp}_n$, and we compute that 
\[H^*_\rm{sing}(\rm{SU}_{2n}/\rm{Sp}_n;\bb{R}) \cong \Lambda^* \langle \ol{c}_{5},\ldots,\ol{c}_{4n-3} \rangle\]
where $\ol{c}_{4i-3}$ has degree $4i-3$.

\begin{proof}[Proof of \cref{thm:monod} for the infinite family AII] One argues as in the case AI and observes that $\Theta_\rm{u}$ acts on the first Chern class by $c_1 \mapsto -c_1$ as conjugation $J \in \rm{U}_{2n}$ acts trivially on cohomology of $\rm{BU}_{2n}$. We leave further details to the reader. \end{proof}

\subsubsection{The infinite family BDI (for $p,q$ both odd)} This is the infinite family of real simple Lie algebras given by $\fr{so}_{p,q}$ with associated family of noncompact connected real Lie groups with finite centre given by the identity component $\rm{SO}_{p,q}^0$ of
\[\rm{SO}_{p,q} = \{A \in \SL_{p+q}(\bb{R}) \mid I_{p,q}A = AI_{p,q}\} \qquad \text{with $I_{p,q} = \begin{bsmallmatrix} -\id_p & 0 \\ 0 & \id_q \end{bsmallmatrix}$.}\]
It has a standard maximal compact group $\rm{SO}_p \times \rm{SO}_q$ and compact dual group $\rm{SO}_{p+q}$ with Cartan involution given by $A \mapsto I_{p,q}AI_{p,q}^{-1}$ \cite[p.~349]{helgason2001differential}. Thus the compact dual symmetric space is $\rm{SO}_{p+q}/(\rm{SO}_p \times \rm{SO}_q)$ and we compute that \cite[p.~320]{Takeuchi}
\[H^*_\rm{sing}(\rm{SO}_{p+q}/(\rm{SO}_p \times \rm{SO}_q);\bb{Q}) \cong \frac{\bb{Q}[p_1,\ldots,p_a,p'_1,\ldots,p'_b]}{(p(\gamma)p(\gamma')-1)} \otimes \Lambda^* \langle \ol{e}_{p+q-1} \rangle.\]
Here $p=2a+1$ and $q=2b+1$, $\gamma$ and $\gamma'$ are the canonical $p$- and $q$-dimensional vector bundles induced by the map $\rm{SO}_{p+q}/(\rm{SO}_p \times \rm{SO}_q) \to \rm{BSO}_p \times \rm{BSO}_q$, $p$ is the total Pontryagin class, the degree of $p_i$ is $4i$, the degree of $p'_j$ is $4j$. Finally, $\ol{e}_{p+q-1}$ has degree $p+q-1$ and it is the unique class whose pullback to $H^{p+q-1}(\rm{SO}(p+q);\bb{R})$ is the same-named unique primitive class whose transgression is the Euler class $e_{p+q} \in H^{p+q}(\rm{BSO}_{p+q};\bb{R})$ up to decomposables.

\begin{proof}[Proof of \cref{thm:monod} for the infinite family BDI (for $p$, $q$ both odd)] The $p_i$ and $p'_j$ are in the image of the geometric map \eqref{eqn:geometric-map} so represented by a bounded continuous cocycle by \cite[Propositions 2.1, 2.2]{HartnickOtt}. Since \cref{thm:monod} is true for $\rm{SO}_{p,q}(\bb{C})$, it suffices by naturality to verify that $\ol{e}_{p+q-1}$ is in the image of the map induced on continuous cohomology by the inclusion $\rm{SO}_{p,q} \to \rm{SO}_{p+q}(\bb{C})$ into the complexification. Arguing as in the case of AI, it suffices to prove that the Cartan involution $\Theta_\rm{u}$ acts by $-1$ on $\ol{e}_{p+q-1}$. By naturality of transgression, this follows because conjugation by $I_{p,q}$ acts by $-1$  on the Euler class $e_{p,q} \in H^{p+q}(\rm{BSO}_{p+q};\bb{R})$ since $p$ is odd and thus $I_{p,q}$ reverses orientation.\end{proof}

\subsubsection{The exceptional case EI} This is given by the real simple Lie algebra $\fr{e}_{6(-6)}$ (the split form of $\fr{e}_6$) and has an associated noncompact real Lie group with finite centre $E_{6(-6)}$. It is given by determinant-preserving automorphisms of the exceptional split Jordan algebra $J_3(\bb{C}')$ \cite[Section 3.15]{Yokota}. It has a maximal compact subgroup $\rm{PSp}_4$ and compact dual group $E_6$. The associated compact dual symmetric space $EI = E_6/\rm{PSp}_4$, and satisfies \cite[p.~322]{Takeuchi} \cite[Theorem 2]{Ishitoya}
\[H^*(EI;\bb{R}) \cong \Lambda^* \langle x_9,x_{17} \rangle \otimes \bb{R}[u_8]/(u_8^3)\]
with classes in the degrees as indicated by the subscripts.

\begin{proof}[Proof of \cref{thm:monod} for the exceptional case EI]
By the cited references, $u_8$ is a pullback of a class in $H^8(\rm{BPSp}_4;\bb{R})$, hence in the image of the geometric map \eqref{eqn:geometric-map} so represented by a bounded continuous cocycle by \cite[Propositions 2.1, 2.2]{HartnickOtt}. Arguing as in the case AI, it suffices to prove that the Cartan involution $\Theta_\rm{u}$ acts on the primitive elements $x_9$ and $x_{17}$ by $-1$. These classes transgress to generators $y_{10}$ and $y_{18}$ in degree $10$ and $18$ in $H^*(\rm{BE}_6;\bb{R})$. Tensoring with $\bb{C}$, the latter can be computed as $S^*(\fr{g}^\vee[2])^{W(E_6)}$. The Cartan involution $\Theta_\rm{u}$ must be a nontrivial outer automorphism \cite[IX.3.1]{helgason2001differential}, and by \cite[p.~220]{Bourbaki} the Weyl group and a nontrivial outer automorphism generate the same group as the Weyl group and $-\id_{\fr{g}}$. Thus $\Theta_\rm{u}$ acts by $(-1)^i$ in degree $2i$, so by $-1$ on $y_{10}$ and $y_{18}$ and by naturality of the transgression by $-1$ on $x_9$ and $x_{17}$.\end{proof}

\subsubsection{The exceptional case EIV} This is the real simple Lie algebra $\fr{e}_{6(-26)}$, with an associated noncompact connected real Lie group with finite centre $E_{6(-26)}$. It is given by determinant-preserving automorphisms of the exceptional Jordan algebra $J_3(\bb{C})$ \cite[Section 3.15]{Yokota}. It has a maximal compact subgroup $F_4$ and compact dual group $E_6$. The associated compact dual symmetric space, also denoted EIV and given by $E_6/F_4$, and satisfies \cite[Proposition 2.5]{Araki} \cite[p.~322]{Takeuchi}
\[H^*(EIV;\bb{R}) \cong \Lambda^* \langle x_9,x_{17} \rangle,\]
with classes in the degrees as indicated by the subscripts. 

\begin{proof}[Proof of \cref{thm:monod} for the exceptional case EIV] One argues exactly as for EI.\end{proof}

\appendix

\section{Goncharov's extension criterion} \label{sec:extension} Our goal is to explain in full detail the proof of an extension criterion originally due to Goncharov extending a construction of Suslin, making some minor corrections to the statements and arguments. For simplicity we will work with rational coefficients in this section (any field of characteristic zero will do) and suppress this from the notation for brevity where possible.

\medskip

For a finite-dimensional vector space $V$ over $F$, consider the chain complex $\widetilde{C}_*(V)$ with entries given by the free $\bb{Q}$-vector space
\[\widetilde{C}_p(V) \coloneq \bb{Q}[(v_0,\compactldots,v_p) \in V^{p+1} \mid v_0,\ldots,v_p \text{ are in general position}]\]
on those ordered $(p+1)$-tuples such that any $k$-tuple for $k \leq \dim(V)$ is linearly independent, and differential given by $d(v_0,\compactldots,v_p) = \sum_{i=0}^p (-1)^i (v_0,\compactldots,\widehat{v}_i,\ldots,v_p)$. This is quasi-isomorphic to $\bb{Q}$ by the standard augmentation $(v_0) \mapsto 1$ and has a natural action of $\GL(V)$, so there are maps unique up to chain homotopy
\begin{equation}\label{eqn:iota-v}\iota_V \colon C_*(\GL(V)) \overset{\simeq}\lra (\widetilde{C}_*(V))_{h\GL(V)} \lra  (\widetilde{C}_*(V))_{\GL(V)} \eqcolon C_*(V)\end{equation}
where the left map is a quasi-isomorphism produced by standard techniques in homological algebra and the right map takes homotopy coinvariants to strict coinvariants. We will refer to the chain complex $C_*(V)$ of coinvariants as the \emph{complex of configurations}, and we will denote the image $(v_0,\compactldots,v_p) \in \smash{\widetilde{C}_*(V)}$ as $(\!( v_0,\compactldots,v_p )\!) \in C_*(V)$. Observe that if $V \cong W$ then $C_*(V)$ is canonically isomorphic to $C_*(W)$, so we may as well consider $C_*(n) \coloneq C_*(F^n)$. We will later see, in \cref{cor:formulas-chain-to-bigrassmannian}, that we may choose the map $\iota_V$ to be given by 
\begin{align*}F_0 \colon C_*(\GL(V)) & \lra C_*(V) \\
[g_1|\compactcdots|g_p] &\longmapsto (\!(v,g_1 \cdot v,\compactldots,g_1 \compactcdots g_p \cdot v)\!)\end{align*} 
for a fixed nonzero vector $v$, on those $p$-tuples of group elements so that the vectors on the right are in general position; we will not need to know its values on other $p$-tuples.

\smallskip

Fix a $\bb{Q}$-vector space $\ds{k}$. A linear map $\sf{h} \colon C_p(n) \to \ds{k}$ defines a map of chain complexes $C_*(n) \to \ds{k}[p]$ when it satisfies the cocycle equation $\sf{h} \circ d = 0$, and by precomposition with $\iota_V$ defines a cohomology class $h = [\sf{h} \circ \iota_{F^n}] \in H^p(\GL_n(F);\ds{k})$. In this section we explain a condition under which we can extend this to a \emph{stable} class: there exist cohomology classes $h^{(m)} \in H^p(\GL_{m+n}(F);\ds{k})$ for $m \geq 0$ so that the restriction map satisfies
\begin{align*}H^p(\GL_{m+n}(F);\ds{k}) &\lra H^p(\GL_{n}(F);\ds{k}) \\
h^{(m)} &\longmapsto h^{(0)} = h.\end{align*} 
We expect that $h^{(m)}$ restricts to $h^{(m-1)}$ and hence all assemble to a class in $H^p(\GL_\infty(F);\ds{k})$, but do not prove this. Moreover, we give an explicit cocycle representative, at least on suitably generic $p$-tuples of elements in $\GL_{m+n}(F)$.

\medskip

We write $\langle - \rangle$ for linear span and $(\!(\langle v_i \rangle \mid v_0,\compactldots,\widehat{v}_i,\compactldots,v_p)\!)$ for the element of $C_{p-1}(n-1)$ given by $(\!(\pi_i(v_0),\compactldots,\pi_i(\widehat{v}_i),\compactldots,\pi_i(v_p))\!) \in \smash{\widetilde{C}_{p-1}(V/\langle v_i \rangle))}$ with $\pi_i \colon V \to V/\langle v_i \rangle$ the projection. In terms of this we define a second ``dual'' differential on $C_*(n)$ by
\[d^\vee(\!(v_0,\ldots,v_p)\!) = \sum_{i=0}^p (-1)^i (\!( \langle v_i \rangle \mid v_0,\compactldots,\widehat{v}_i,\compactldots,v_p)\!).\]

\begin{definition}A linear map $\sf{h} \colon C_*(n) \to \ds{k}[p]$ is a \emph{bi-Grassmannian} cocycle if it satisfies $\sf{h} \circ d = 0$ and $\sf{h} \circ d^\vee = 0$.\end{definition}

That is, it is a map that satisfies a cocycle equation as well as a ``dual cocycle equation.'' The following is essentially due to Goncharov \cite[Section 2.6]{Gon95b} \cite[Section 4]{Gon16}, generalising work of Suslin \cite[\S 3]{Sus90}. It uses the notion of a \emph{partial $k$-frame}, which is an ordered $k$-tuple $v^\bullet = (v^1,\ldots,v^k)$ of linearly independent vectors. We say a collection of $m$ partial $k$-frames is \emph{in general position} if the collection of all $mk$ vectors is in general position.

\begin{theorem} \label{thm:extension-criterion}
    Let $F$ be an infinite field and $\sf{h} \colon C_*(n) \to \ds{k}[p]$ be a bi-Grassmannian cocycle.
    \begin{enumerate}[\noindent (i)]
        \item The cohomology class $h \in H^p(\GL_n(F);\ds{k})$ defined by the cocycle $C_*(\GL_n(F)) \to C_*(n) \overset{\sf{h}}\to \ds{k}[p]$ is stable: for any $m \ge 1$ there exists a cohomology class $h^{(m)} \in H^p(\GL_{m+n}(F);\ds{k})$ which restricts to $h^{(0)} = h$ along the inclusion $\GL_n(F) \to \GL_{m+n}(F)$.
        \item For a fixed partial $(m+1)$-frame $v^\bullet$ in $F^{m+n}$, $h^{(m)}$ can be represented by an inhomogeneous cocycle satisfying
    \begin{align*} \sf{h}^{(m)} \colon \GL_{m+n}(F)^p &\lra \ds{k} \\
    [g_1|\compactcdots|g_p] &\longmapsto (-1)^{m}\sum_{\substack{
i_0+\compactcdots+i_p=m\\
i_0,\compactldots, i_p\ge 0
}} \sf{h}\Big((\!( V_{i_0,\ldots,i_p} \mid v_0^{i_0+1},\compactldots, v_p^{i_p+1})\!)\Big)\end{align*}
for $p$-tuples defined as $v_0^\bullet \coloneq v^\bullet$ and $v_i^\bullet \coloneq (g_1 \compactcdots g_i) \cdot v^\bullet$ for $1 \le i \le p$ and subspace $V_{i_0,\ldots,i_p} \coloneqq \langle v_0^1,\compactldots,v_0^{i_0},\compactldots,v_p^1,\compactldots,v_p^{i_p} \rangle$, as long as the $p$-tuples are in general position.
    \end{enumerate}
\end{theorem}

The proof of this extension theorem is a construction, and this uses the \emph{($n$-truncated) bi-Grassmannian bicomplex}, denoted $C_{*,*}(n)$, given by
\vspace{-.4cm} \[
    \begin{tikzcd}
      &  \vdots \dar & \vdots \dar & \vdots \dar \\[-8pt] 
      \compactcdots \rar & C_{n+4}(n+2)  \rar{d} \dar{d^\vee} & C_{n+3}(n+2)  \rar{d} \dar{d^\vee} & C_{n+2}(n+2)  \dar{d^\vee} \\[-5pt]
      \compactcdots \rar & C_{n+3}(n+1)  \rar{d} \dar{d^\vee} & C_{n+2}(n+1)  \rar{d} \dar{d^\vee} & C_{n+1}(n+1)  \dar{d^\vee} \\[-5pt]
      \compactcdots \rar & C_{n+2}(n)  \rar{d} & C_{n+1}(n)  \rar{d} & C_{n}(n).
    \end{tikzcd}\]
Its totalisation is the \emph{(n-truncated) bi-Grassmannian complex}, denoted $(BC_*(n),\partial)$, given by
\[{BC}_p(n) \coloneq \bigoplus_{i=0}^{p-n} C_{p}(n+i) \quad \text{with} \quad \partial=d+d^\vee.\]
Observe that $\sf{h}$ is a bi-Grassmannian cocycle if and only if extends uniquely to a chain map $BC_*(n) \to \ds{k}[p]$. Then the construction is as follows:
\begin{enumerate}[\noindent  1.]
    \item For each $m \ge 0$ we construct a complex $BD^m_*(m+n)$ as the totalisation of a certain bicomplex as well as 
    \begin{enumerate}[(a)]
        \item explicit chain maps $\phi_n^m \colon BD^m_*(m+n) \to BC_*(n)$, and
        \item chain maps $F_m \colon C_*(\GL_{m+n}(F)) \to BD^{m}_*(m+n)$ unique up to chain homotopy. 
    \end{enumerate}
    By composing these one gets for each $m \ge 0$ a chain map unique up to chain homotopy
    \[\sigma_m \coloneq \phi_n^{m} \circ F_m \colon C_*(\GL_{m+n}(F)) \lra BC_*(n).\]
    \item We prove that the maps $\sigma_{m+n}$ are compatible with stabilisation, in the sense that for $m \geq 0$
    \begin{center}
        \begin{tikzcd}
            C_*(\GL_n(F)) \rar{\sigma_0} \dar[swap]{s} & BC_*(n) \dar[equal] \\[-5pt]
            C_*(\GL_{n+m}(F)) \rar{\sigma_m} & BC_*(n)
        \end{tikzcd}
    \end{center}
    commutes up to chain homotopy, where $s$ denotes the map induced by the stabilisation map $\GL_n(F) \hookrightarrow \GL_{n+m}(F)$ induces by the inclusion $F^n \to F^m \oplus F^n$.
    \item We find an explicit expression for $F_m \colon C_*(\GL_{m+n}(F)) \to BD^{m}_*(m+n)$ in the generic case, and hence for $\sigma_m$, by factoring over a chain complex $F^m_*(m+n)$ of partial frames. 
    \item We observe that for $m=0$, $\phi_n^0$ will be given by the map
\[\inc \circ \tau_{\geq n} \colon C_*(n) \lra C_{* \ge n}(n) \lra BC_*(n)\]
of truncation followed by the inclusion of the bottom row. The map $F_0 \colon C_*(\GL_n(F)) \to C_*(n)$ was already discussed above and thus $h^{(0)}=h$, at least on generic $p$-tuples.
\end{enumerate}

\subsection{From inhomogeneous chains to the bi-Grassmannian complex}
The goal of this subsection is to define the chain maps $\sigma_m \colon C_*(\GL_{m+n}(F)) \to BC_*(n)$ for $m \ge 0$ and their compatibility with stabilisation. This establishes steps 1 and 2 of the construction.

\subsubsection{Group homology constructions} \label{sec:standard-material-resolutions} In this subsection we recall some basic facts about group homology. We work in the $1$-category of chain complexes of left $\bb{Q}[G]$-modules, which we abbreviate to $G$-modules. For any chain complex $A_*$ of $G$-modules there is a canonical map
\[\tau_A \colon (A_*)_{hG} \lra (A_*)_G\]
from the homotopy $G$-coinvariants to the strict $G$-coinvariants. The former will be explicitly modelled by the total complex of the bicomplex $B_*(\bb{Q},\bb{Q}[G],A_*)$ of inhomogeneous chains with coefficients in $A_*$.

Suppose there is a chain map $\epsilon \colon A_* \to \bb{Q}$ inducing an isomorphism on homology, i.e.~$A_*$ is a \emph{resolution} of $\bb{Q}$ and $\epsilon$ is an \emph{augmentation} of $A_*$. A standard construction of a resolution of $\bb{Q}$ is the bar construction $B_*(\bb{Q}[G],\bb{Q}[G],\bb{Q}) \to \bb{Q}$. This consists levelwise of free $\bb{Q}[G]$-modules, so standard homological algebra says there exists a map of chain complexes
\[B_*(\bb{Q}[G],\bb{Q}[G],\bb{Q}) \lra A_*\]
covering the identity on $\bb{Q}$, unique up to chain homotopy. Taking $G$-coinvariants, we get a map of chain complexes
\[C_*(G) = B_*(\bb{Q}[G],\bb{Q}[G],\bb{Q})_G \lra (A_*)_G\]
of rational vector spaces, also unique up to chain homotopy, where the domain is the usual inhomogeneous chains computing group homology (with rational coefficients).

Let us next discuss how natural this map is. Let $\varphi \colon G \to G'$ be a homomorphism, $A_*$ be a chain complex of $G$-modules with augmentation $\epsilon \colon A_* \to \bb{Q}$, $A'_*$ be a chain complex of $G'$-modules with augmentation $\epsilon' \colon A'_* \to \bb{Q}$, and $f \colon A_* \to A'_*$ be a map of augmented chain complexes of $G$-modules (where the target is considered as such by restriction along $\varphi$). Then the following diagram commutates up to chain homotopy
\[\begin{tikzcd} C_*(G) \rar \dar[swap]{\varphi_*} & (A_*)_G \dar{f_G} \\[-5pt]
C_*(G') \rar & (A'_*)_G.\end{tikzcd}\]
To see this, note that both compositions 
\[B_*(\bb{Q}[G],\bb{Q}[G],\bb{Q}) \to B_*(\bb{Q}[G'],\bb{Q}[G'],\bb{Q}) \to A'_* \quad \text{and} \quad B_*(\bb{Q}[G],\bb{Q}[G],\bb{Q}) \to A_* \to A'_*\] 
are maps covering the identity on $\bb{Q}$ and invoke uniqueness up to chain homotopy. Now taking appropriate $G$- and $G'$-coinvariants gives the result.

\subsubsection{Construction of a new bicomplex} Our next goal is the construction of the bicomplex $\widetilde{D}_{*,*}(V)$, which involves partial shuffle and symmetrisation maps. Firstly, for $0 \le k \le p$ we define maps
\begin{align*}\delta^{(k)} \colon \widetilde{C}_p(V) &\lra \widetilde{C}_{p-1}(V) \\    
(v_0,\compactldots,v_p) &\longmapsto \sum_{i=0}^{p-k} (-1)^i (v_0,\compactldots,\widehat{v}_{k+i},\compactldots,v_p).
\end{align*}
In particular, $\delta^{(0)}$ is the usual bar differential $d$. These satisfy $\delta^{(k)} \circ \delta^{(k)}=0$ and because $F$ is an infinite field we have \cite[Lemmas 2.11, 2.12]{Gon95b}:

\begin{lemma} \label{lem acyclic}
    The complexes
\[\compactldots \lra \widetilde{C}_{k+1}(V) \xrightarrow{\delta^{(k)}} \widetilde{C}_{k}(V) \xrightarrow{\delta^{(k)}} \widetilde{C}_{k-1}(V)\]
are acyclic for any $k \ge 0$, where $\widetilde{C}_{-1}(V) \coloneq \bb{Q}$, viewed as the free $\bb{Q}$-module on the empty tuple, and $\delta^{(0)}(v_0)=1$. 
\end{lemma}

Secondly, for $0 \le k \le p+1$ we define maps
\begin{align*} \rm{Sym}_k \colon \widetilde{C}_p(V) &\lra \widetilde{C}_p(V) \\
(v_0,\compactldots,v_p) &\longmapsto \frac{1}{k!} \sum_{\sigma \in S_k} (v_{\sigma(0)},\compactldots,v_{\sigma(k-1)},v_k,\compactldots,v_p)\end{align*}
by symmetrizing the first $k$ vectors with the convention that $\rm{Sym}_0=\id$, and maps
\begin{align*} \lambda^{(k)} \colon \widetilde{C}_p(V) &\lra \widetilde{C}_p(V) \\
(v_0,\compactldots,v_p) &\longmapsto \sum_{i=0}^{p-k} (-1)^i \rm{Sym}_{k+1}(v_0,\compactldots,v_{k-1},v_{k+i},v_{k},\compactldots,\widehat{v_{k+i}},\compactldots, v_p)\end{align*}
moving the later vectors to the $k$th position, with appropriate sign. These satisfy $\delta^{(k+1)} \circ \lambda^{(k)} + \lambda^{(k)} \circ \delta^{(k)}=0$ for $k \ge 0$, and $\lambda^{(k+1)} \circ \lambda^{(k)}=0$ for $k \ge 1$ \cite[Lemma 2.3, 2.14]{Gon95b}. The first equation holds before symmetrising but the second one does not. 

Now, for each $m \ge 0$ we define a $\bb{Q}[\GL(V)]$-bicomplex $\widetilde{D}^m_{*,*}(V)$ via 
\[
    \begin{tikzcd}[column sep=small, row sep=small]
        \cdots \rar & \widetilde{C}_3(V) \rar{\delta^{(0)}} \dar{\lambda^{(0)}} & \widetilde{C}_2(V) \rar{\delta^{(0)}} \dar{\lambda^{(0)}} & \widetilde{C}_1(V) \rar{\delta^{(0)}} \dar{\lambda^{(0)}} & \widetilde{C}_0(V) \dar{\lambda^{(0)}} \\
        \cdots \rar & \widetilde{C}_3(V) \rar{\delta^{(1)}} \dar{\lambda^{(1)}} & \widetilde{C}_2(V) \rar{\delta^{(1)}} \dar{\lambda^{(1)}} & \widetilde{C}_1(V) \rar{\delta^{(1)}} \dar{\lambda^{(1)}} & \widetilde{C}_0(V)  \\
        \cdots \rar & \widetilde{C}_3(V) \rar{\delta^{(2)}} \dar{\lambda^{(2)}} & \widetilde{C}_2(V) \rar{\delta^{(2)}} \dar{\lambda^{(2)}} & \widetilde{C}_1(V) \\
         \cdots \rar & \widetilde{C}_3(V) \rar{\delta^{(3)}} \dar{\lambda^{(3)}} & \widetilde{C}_2(V) \\
         & \cdots & 
    \end{tikzcd}
\]
ending in the row with horizontal differential $\delta^{(m)}$, so this bicomplex has precisely $m+1$ rows: the \emph{$k$th row} is the one where the horizontal differential is $\delta^{(k)}$, for $0 \le k \le m$. For example, the last row ends with $\tilde{C}_{m}(V)$ consisting of $(m+1)$-tuples of vectors. The bigrading is such that $\smash{\widetilde{D}^m_{-k,p}}$ is given by the term $\smash{\widetilde{C}_p(V)}$ in the $k$th row and we find it illuminating to denote its generators as $(v_0,\compactldots,v_{k-1} \vert v_k,\compactldots,v_p)$. Note that $k=0$ is allowed, yielding elements $(\vert v_0,\compactldots,v_p)$ and that $k-1 = p$ is allowed, yielding elements $(v_0,\compactldots,v_p \vert )$. 

Its totalisation is denoted as follows
\[\widetilde{BD}^m_d(V) \coloneq \bigoplus_{p-k=d} \widetilde{D}_{-k,p}^m(V)\]
with differential $\widetilde{d}=\delta^{(k)}+\lambda^{(k)}$. That is, we put $\widetilde{C}_p(V)$ in the $k$th row in degree $p-k$, and the complex is concentrated in degrees $\ge -1$, and both $\delta^{(k)}$ and $\lambda^{(k)}$ lower degree by $1$. It admits an augmentation given by projection to the top-right term $\smash{\widetilde{D}_{0,0}(V)} = \smash{\widetilde{C}_0(V)}$ and then using the standard augmentation $\widetilde{C}_0(V) \ni (\vert v) \mapsto 1 \in \bb{Q}$. This is a homology isomorphism by \cref{lem acyclic}, so we have maps unique up to chain homotopy
\begin{equation}\label{eqn:iota-m-v} \iota^m_V \colon C_*(\GL(V)) \overset{\simeq}\lra (\widetilde{BD}^m_*(V))_{h\GL(V)} \lra  (\widetilde{BD}^m_*(V))_{\GL(V)} \eqcolon BD^m_*(V)\end{equation}
where the left is a quasi-isomorphism and the right map takes homotopy coinvariants to strict coinvariants. We will denote its generators of the latter as $(\!(v_0,\compactldots,v_{k-1} \vert v_k,\compactldots,v_p)\!)$. If $V \cong V'$ then $BD^m_{*}(V)$ and $BD^m_{*}(V')$ are canonically isomorphic; we use the notation $BD^m_{*}(n) \coloneq BD^m_{*}(F^n)$.

\subsubsection{Maps to the bi-Grassmannian complex} We will first recall a construction of a map from this new complex to the bi-Grassmannian complex. To do so, we let $D^m_{*,*}(n)$ denote the coinvariants of the bicomplex $\smash{\widetilde{D}^m_{*,*}(F^n)}$. The discussion in \cite[p.~260]{Gon95b} yields the following:

\begin{lemma} \label{lem map to bi grassmannian complex}
    There is a map of bicomplexes 
    \begin{align*}\phi_n^m \colon D^m_{*,*}(m+n) &\lra C_{*,*}(n) \\
    (\!(v_0,\compactldots,v_{k-1} \vert v_k,\compactldots, v_p)\!) &\longmapsto \begin{cases}(\!(\langle v_0,\compactldots,v_{k-1}\rangle \vert v_k,\compactldots,v_p)\!) & \text{for $p \ge n+m$,} \\
    0 & \text{otherwise,}\end{cases}\end{align*} 
    where the left side is an element of $D^m_{-k,p}(n+m)$ and the right side is an element of $C_{p-k}(n+m-k)$.\footnote{The map is well-defined because if $C_p(n+m)$ is in the $k$th row then $0 \le k \le m$ so $n+m-k \ge n$. Also $p-k \ge n+m-k$ or otherwise the map is $0$.} Taking totalisations this yields a chain map
    \[\phi^m_n \colon BD^m_*(n+m) \lra BC_*(n).\]
\end{lemma} 

Using this, we can now define the $\sigma_m$, completing step 1.

\begin{definition}For $m \ge 0$, we define $\sigma_m \colon C_*(\GL_{m+n}(F)) \to BC_*(n)$ as the composition
\[C_*(\GL_{m+n}(F)) \xrightarrow{\iota_{m+n}^{m}}  BD^{m}_*(m+n) \xrightarrow{\phi^m_n} BC_*(n),\]
which is well-defined up to homotopy.\end{definition}

\begin{example}When $m=0$ we have $\widetilde{BD}^0_{*}(n) = \widetilde{C}_*(n)$ with augmentation the standard one, the map $\phi^0_n \colon C_*(n) \to BC_*(n)$ is the inclusion of the truncation of the bottom row, and $\sigma_0$ is defined as $\phi^0_n \circ \iota_V$ up to chain homotopy.\end{example}

\subsubsection{Stability} We now define an iterated stabilisation map. Its definition uses the following element, described in terms of the standard basis $e_1,\compactldots,e_m$ of $F^m$:

\begin{lemma}\label{lem:e-element}The element $E \coloneq \sum_{k=0}^{m-1} (-1)^{k} \rm{Sym}_{k}(e_1,\compactldots,e_k \vert e_{k+1}) \in \oplus_{k=0}^{m-1} \widetilde{D}^{m}_{-k,k}(F^m \oplus V)$ satisfies
\[\tilde{d}(E) = (-1)^{m-1}\rm{Sym}_m(e_1,\compactldots,e_m \vert ).\] \end{lemma}

\begin{proof}
To prove the claim, observe that
    \begin{align*}\delta^{(k)}(\rm{Sym}_{k}(e_1,\compactldots,e_{k} \vert e_{k+1})) &= \begin{cases} 0 & \text{if $k=0$,} \\
    \rm{Sym}_{k}(e_1,\compactldots,e_{k} \vert ) \in \widetilde{D}^m_{-k,k-1} & \text{if $ 1 \leq k \leq m-1$,}\end{cases} \\
    \lambda^{(k)}(\rm{Sym}_{k}(e_1,\compactldots,e_{k} \vert e_{k+1}))&= \rm{Sym}_{k+1}(e_1,\compactldots,e_{k},e_{k+1} \vert ) \in \widetilde{D}^m_{-k-1,k}\quad \text{if $0 \le k \le m-1$}\end{align*}
    where the latter uses that $\rm{Sym}_{k+1} \circ \rm{Sym}_{k}=\rm{Sym}_{k+1}$. Thus, all terms cancel because of the extra signs in the definition of $E$, except $\lambda^{(m-1)}((-1)^{m-1} \rm{Sym}_{m-1}(e_1,\compactldots,e_{m-1}\vert e_m)) = (-1)^{m-1}\rm{Sym}_m(e_1,\compactldots,e_m\vert)$. \end{proof}

To define the stabilisation map we have to slightly modify $\widetilde{D}^m_{*,*}(V)$: take $\widetilde{D}'^m_{*,*}(V)$ to be the bicomplex 
\[\widetilde{D}'^m_{-k,p}(V) \coloneq \begin{cases} \widetilde{D}^m_{-k,p}(V) & \text{if $k < m$,} \\
\widetilde{D}'^m_{-m,p}(V) & \text{if $k=m$},\end{cases}\]
where $\widetilde{D}'^m_{-m,p}(V)$ is the free $\bb{Q}$-vector space on $(p+1)$-tuples $(v_0,\compactldots,v_{m-1} \vert v_m,\compactldots,v_p)$ so that the $v_0,\ldots,v_{m-1}$ are in general position and the images of $v_m,\ldots,v_p$ in $V/\langle v_0,\compactldots,v_m \rangle$ are in general position. Generalising \cref{lem acyclic}, this modified bottom row is also acyclic and thus the inclusion $\widetilde{D}^m_{*,*}(V) \to \widetilde{D}'^m_{*,*}(V)$ is a quasi-isomorphism and by Section \ref{sec:standard-material-resolutions} it induces the right vertical map in  a square commuting up to chain homotopy
\[\begin{tikzcd} C_*(\GL(V)) \dar[equal] \rar & BD_*(V) \dar \\[-5pt]
C_*(\GL(V)) \rar & BD'_*(V).\end{tikzcd}\]
This modification allows us to construct the following map:

\begin{lemma}There is a map of chain complexes
\begin{align*}\widetilde{s}^m \colon \widetilde{C}_{*}(V) &\lra \widetilde{BD}'{}_{*}^m(F^m \oplus V) \\
(v_0,\compactldots,v_p) &\longmapsto \begin{cases} (-1)^m\rm{Sym}_m (e_1,\compactldots,e_m \vert v_0,\compactldots,v_p) & \text{if $p \neq 0$,} \\
(-1)^m \rm{Sym}_m(e_1,\compactldots,e_m \vert v_0) + E & \text{if $p = 0$.} \end{cases}\end{align*}\end{lemma} 

\begin{proof}For the case $*=0$ we have $d(v_0) = 0$ and we verify
\begin{align*}\widetilde{d}(\widetilde{s}^m(v_0)) &= \delta^{(m-1)}((-1)^m \rm{Sym}_m(e_1,\compactldots,e_m \vert v_0)) + \widetilde{d}(E) & \\
&= (-1)^m\rm{Sym}_m(e_1,\compactldots,e_m\vert) + (-1)^{m-1} \rm{Sym}_m(e_1,\compactldots,e_m\vert) = 0.\end{align*}
The cases $*>0$ are easily verified, so this is a chain map.
\end{proof}

This stabilisation map is equivariant for the standard stabilisation map $s^m \colon \GL(V) \to \GL(F^m \oplus V)$ and is compatible with the augmentations by definition. Thus Section \ref{sec:standard-material-resolutions} yields the following diagram commuting up to chain homotopy, where the left horizontal maps are \eqref{eqn:iota-v}, \eqref{eqn:iota-m-v}, and middle vertical map $s^m$ is the map induced by $\widetilde{s}^m$ on covariants:
   \[\begin{tikzcd}
             C_*(\GL_n(F)) \rar{\iota_V} \dar[swap]{C_*(s^m)} &[10 pt] C_*(n) \dar{s^m} \rar{\phi^0_n} & BC_*(n) \dar[equal]\\[-5pt] 
             C_*(\GL_{m+n}(F)) \rar{{\iota'}^m_{F^m \oplus V}} & {BD'}_*^{m}(m+n) \rar{{\phi'}^m_n}  & BC_*(n) \dar[equal] \\[-5pt]
             C_*(\GL_{m+n}(F)) \rar{\iota^m_{F^m \oplus V}} \uar[equal] & BD_*^{m}(m+n) \rar{\phi^m_n} \uar & BC_*(n).
        \end{tikzcd}\]
On the right we have added that the map of \cref{lem map to bi grassmannian complex} extends to ${BD'}_*^m(V)$, and is compatible with stabilisation by inspection. From this we obtain the following square commuting up to chain homotopy, completing the proof of Step 2:
\[\begin{tikzcd}
            C_*(\GL_n(F)) \rar{\sigma_0} \dar[swap]{C_*(s^m)} & BC_*(n) \dar[equal] \\[-5pt]
            C_*(\GL_{m+n}(F)) \rar{\sigma_{m}} & BC_*(n).
        \end{tikzcd} \]

\subsection{An explicit formula} Having completed steps 1 and 2, we now provide the explicit formula of step 3. This proves \cref{thm:extension-criterion}, once we observe that $F^0$ is as stated in step 4.

\subsubsection{The complex of partial frames} Recall a \emph{partial $k$-frame} is a sequence $v^\bullet= (v^1,\compactldots,v^k)$ of linearly independent vectors in $V$. To find an explicit chain map representing 
\[\iota_V^{k-1} \colon C_*(\GL(V)) \lra BD_*^{k-1}(V)\]
as in \eqref{eqn:iota-m-v} we introduce an intermediate complex, inspired by Goncharov's decorated complex \cite[Section 4.1]{Gon16} and constructed from partial $k$-frames in general position.

\medskip

For a finite-dimensional vector space $V$ over $F$, consider the chain complex with entries given by the free $\bb{Q}$-vector space (the superscripts $k-1$ for $k$ frames is unfortunate, but fits the rest of the notation better)
\[\widetilde{F}^{k-1}_p(V) \coloneq \bb{Q}[(v_0^\bullet,\compactldots,v_p^\bullet) \mid v_i^\bullet \text{ are partial $k$-frames in general position}]\]
on ordered $(p+1)$-tuples of partial $k$-frames in general position in $V$, and differential given by $d \colon (v^\bullet_0,\compactldots,v^\bullet_p) = \sum_{i=0}^p (-1)^i (v_0^\bullet,\compactldots,\widehat{v^\bullet_i},\compactldots,v_p^\bullet)$. The following is proven using a standard argument that uses a partial $m$-frame in general position to cone off a cycle, which exists since $F$ is infinite:

\begin{lemma} \label{lem m frames complex acyclic}
    The following is an acyclic complex of $\GL(V)$-modules 
    \[\compactldots \lra \widetilde{F}^{k-1}_2(V) \overset{d}\lra \widetilde{F}_1^{k-1}(V) \overset{d} \lra \widetilde{F}_0^{k-1}(V) \overset{\eta}\lra \bb{Q},\]
    where $\eta$ is the augmentation unique determined by $(v^\bullet) \mapsto 1$. 
\end{lemma}

This has a natural action of $\GL(V)$ and we denote the chain complex of coinvariants by $F^{k-1}_*(V) \coloneq \widetilde{F}^{k-1}_*(V)_{\GL(V)}$. Observe that if $V \cong W$ then $F^{k-1}_*(V)$ is canonically isomorphic to $F^{k-1}_*(V')$, so we may as well consider $F^{k-1}_p(n) \coloneqq F^{k-1}_p(F^n)$. We call this the complex of \emph{partial $k$-frames}. By \cref{lem m frames complex acyclic} and Section \ref{sec:standard-material-resolutions} there is a chain map, well-defined up to chain homotopy,
\[j_V^{k-1} \colon C_*(\GL(V)) \lra F^{k-1}_*(V).\]

\begin{lemma} \label{lem:chain-map-from-partial-frams} For each $k \ge 1$ there is a map of augmented chain complexes
    \begin{align*}\widetilde{T}_* \colon \widetilde{F}^{k-1}_*(V) &\lra \widetilde{BD}^{k-1}_*(V) \\
    (v^\bullet_0,\compactldots,v^\bullet_p) &\longmapsto \bigoplus_{j=0}^{k-1} \sum_{\substack{
i_0+\compactcdots+i_p=j\\
i_0,\compactldots, i_p\ge 0
}}(-1)^j \rm{Sym}_j(v_0^1,\compactldots,v_0^{i_0},\compactldots,v_p^1,\compactldots,v_p^{i_p} \vert v_0^{i_0+1},\compactldots, v_p^{i_p+1}).\end{align*}
\end{lemma} 

\begin{proof} The proof is a lengthy computation. We need to abbreviate ellipses to a single period.

\smallskip

\noindent \emph{Step 1.} We first compute the $(-j,p-1+j)$ component of $(\widetilde{T}_* \circ \widetilde{D})(v_0^\bullet,.,v_p^\bullet)$ for $p\ge 1$. By the formulae, it is given by the following element of $\widetilde{C}_{p-1+j}(V)$
\[(-1)^j \sum_{t=0}^p (-1)^t \sum_{\substack{
i_0+\compactcdots+i_p=j\\
i_0,\compactldots, i_p\ge 0 \\
i_t=0
}} \rm{Sym}_j(v_0^1,.,v_0^{i_0},.,v_p^1,.,v_p^{i_p},v_0^{i_0+1},.,\widehat{v_t^{i_t+1}},.,v_p^{i_p+1}).\]

\smallskip

\noindent \emph{Step 2.} We next compute the $(-j,p-1+j)$-component of $(\widetilde{d} \circ \widetilde{T}_*)(v_0^\bullet,\compactldots,v_p^\bullet)$ for $p \ge 1$. This has two components. For the first, we need to apply $\delta^{(j)}$ to the $(-j,p+j)$-component of $\widetilde{T}_*(v_0^\bullet,.,v_p^\bullet)$, which is given by 
    \[\sum_{\substack{
i_0+\compactcdots+i_p=j\\
i_0,\compactldots, i_p\ge 0
}}(-1)^j \sum_{t=0}^p (-1)^t \rm{Sym}_k(v_0^1,.,v_0^{i_0},.,v_p^1,.,v_p^{i_p},v_0^{i_0+1},., \widehat{v_t^{i_t+1}}, v_p^{i_p+1}). \]
For the second, we need to apply $\lambda^{(j-1)}$ to the $(-j+1,p+j-1)$-component of $\widetilde{T}_*(v_0^\bullet,.,v_p^\bullet)$ (with the convention that this vanishes for $k=0$), which is given by 
\begin{align*}
    &\lambda^{(j-1)} \Big( \sum_{\substack{
i_0+\compactcdots+i_p=j-1\\
i_0,\compactldots, i_p\ge 0
}}(-1)^{k-1} \rm{Sym}_{j-1}(v_0^1,.,v_0^{i_0},.,v_p^1,.,v_p^{i_p},v_0^{i_0+1},., v_p^{i_p+1})\Big) =\\ 
&
 \sum_{\substack{
i_0+\compactcdots+i_p=j-1\\
i_0,\compactldots, i_p\ge 0
}} (-1)^{k-1} \sum_{t=0}^p (-1)^t \rm{Sym}_j(v_0^1,.,v_0^{i_0},.,v_t^1,.,v_t^{i_t+1},.,v_p^1,.,v_p^{i_p},v_0^{i_0+1},.,\widehat{v_t^{i_t+1}},.,v_p^{i_p+1}),
\end{align*}
using that $\rm{Sym}_j \circ \rm{Sym}_{j-1}=\rm{Sym}_j$ in the last equality. 
By renaming $i_t+1$ to $i_t$, with condition that $i_t \ge 1$, we can now write the last expression as 
\[-\sum_{\substack{
i_0+\compactcdots+i_p=h\\
i_0,\compactldots, i_p\ge 0 \\
i_t \ge 1
}}(-1)^k \sum_{t=0}^p (-1)^t \rm{Sym}_j(v_0^1,\compactldots,v_0^{i_0},\compactldots,v_p^1,\compactldots,v_p^{i_p},v_0^{i_0+1},\compactldots, \widehat{v_t^{i_t+1}}, v_p^{i_p+1}). \]
Adding these two cases, the part where $i_t \ge 1$ cancels out and we obtain the following element of $\widetilde{C}_{p-1+j}(V)$: 
\[(-1)^j \sum_{t=0}^p (-1)^t \sum_{\substack{
i_0+\compactcdots+i_p=j\\
i_0,\compactldots, i_p\ge 0 \\
i_t=0
}} \rm{Sym}_j(v_0^1,.,v_0^{i_0},.,v_p^1,.,v_p^{i_p},v_0^{i_0+1},.,\widehat{v_t^{i_t+1}},.,v_p^{i_p+1}).\]

\smallskip

\noindent \emph{Step 3.} We next prove that $\widetilde{T}_*$ is a chain map.  Since $\widetilde{F}^{k-1}_*(V)=0$ for $* <0$ it suffices to verify $\widetilde{d} \circ \widetilde{T}_* = \widetilde{T}_* \circ \widetilde{D}$ on $\widetilde{F}^{k-1}_{* \ge 1}(V)$ and that $\widetilde{d} \circ \widetilde{T}$ vanishes on $\widetilde{F}^{k-1}_0(V)$.   The first case follows by direct comparison of Steps 1 and 2. The second case follows from the fact that for $(v_0^\bullet) \in \widetilde{F}^{k-1}_0(V)$ we have 
\[\widetilde{T}(v_0^\bullet)= \bigoplus_{j=0}^{k-1} (-1)^j  \rm{Sym}_j(v_0^1,.,v_0^{j+1}) \in \bigoplus_{k=0}^{k-1}\widetilde{D}^{k-1}_{-j,j}(V)=\widetilde{BD}^{k-1}_0(V),\]
and this is a cycle for $\widetilde{d}$ by the computation of \cref{lem:e-element}.

\smallskip

\noindent \emph{Step 4.} We finally verify compatibility with augmentations: we need to show that for $(v_0^\bullet) \in \widetilde{F}^{k-1}_0(V)$ we have $(\epsilon \circ \widetilde{T}_*)(v_0)=1$. As before, for $(v_0^\bullet) \in \widetilde{F}^{k-1}_0(V)$ we have 
\[\widetilde{T}(v_0^\bullet)= \bigoplus_{j=0}^{k-1} (-1)^j  \rm{Sym}_j(v_0^1,\compactldots,v_0^{j+1}) \in \bigoplus_{j=0}^{k-1}\widetilde{D}^{k-1}_{-j,j}(V)=\widetilde{BD}^{k-1}_0(V),\]
and hence $\epsilon(\widetilde{T}_*(v_0^\bullet))=1$. 
\end{proof}

Taking $\GL(V)$-coinvariants we obtain a map $T_* \colon F^{k-1}_*(n) \to BD^{k-1}_*(n)$ that by Section \ref{sec:standard-material-resolutions} has the property that the following is homotopic to $\iota_V^{k-1}$
\[C_*(\GL(V)) \xrightarrow{j_V^{k-1}} F^{k-1}_*(V) \xrightarrow{T_*} BD^{k-1}_*(V).\]

\begin{example}
    A partial $1$-frame is the same as a non-zero vector, $\widetilde{F}^0_*(V)=\widetilde{C}_*(V)$, and $F^0_*(n) = C_*(n)$. The map $T_* \colon F^0_*(V) \to BD^0_*(V)$ is the identity by identifying both the source and target canonically with $C_*(V)$. 
\end{example}

\subsubsection{An explicit homotopy inverse} To finish the proof of \cref{thm:extension-criterion}, it remains to produce an explicit formula for a chain map
\[j_V^{k-1} \colon C_*(\GL(V)) \lra F^{k-1}_*(V).\]
Fix the data of a partial $k$-frame $v^\bullet$ in $V$ once and for all. We then say that a tuple $(g_1,\compactldots,g_p) \in \GL(V)^p$ is $v^\bullet$-\emph{generic} if $(v^\bullet, g_1  \cdot v^\bullet, \compactldots, (g_1 \compactcdots g_p) \cdot v^\bullet)$ is a $(p+1)$-tuple of partial $k$-frames in general position. Note that if $(g_1,\compactldots,g_p)$ is $v^\bullet$-generic then $(g_1  \cdot v^\bullet, \compactldots, (g_1 \compactcdots g_p) \cdot v^\bullet)$ is a $p$-tuple of partial $k$-frames in general position, and hence so are $(v^\bullet, g_2  \cdot v^\bullet, \compactldots, (g_2 \compactcdots g_p) \cdot v^\bullet)$. Thus, $(g_2,\compactldots,g_p) \in \GL(V)^p$ is also $v^\bullet$-generic. 

\begin{proposition} \label{prop partial map from group chains to partial frames}
The chain map $j_V^{k-1} \colon C_*(\GL(V)) \to F^{k-1}_*(V)$ can be chosen to be given by 
\[[g_1|\compactcdots|g_p] \longmapsto (\!(v^\bullet,g_1 \cdot v^\bullet, \compactldots, (g_1\compactcdots g_p) \cdot v^\bullet)\!), \]
on $v^\bullet$-generic generators $(g_1,\compactldots,g_p) \in C_p'(\GL(V))$. 
\end{proposition}

\begin{proof}
By Section \ref{sec:standard-material-resolutions}, this map is obtained choosing a chain map of $\GL(V)$-modules
\[\widetilde{j}_V^{k-1} \colon B_*(\bb{Q}[\GL(V)],\bb{Q}[\GL(V)],\bb{Q}) \to \widetilde{F}^{k-1}_*(V)\]
covering the identity on $\bb{Q}$, and taking strict coinvariants. The proof of the first step uses that the domain is levelwise free given in degree $p$ by $\free_{\bb{Q}[\GL(V)]}(\bb{Q}[\GL(V)^p])$, so such a chain map can be produced inductively. For the initial case, in degree $0$ we produce a map 
\[(\widetilde{j}_V^{k-1})_0 \colon \free_{\bb{Q}[\GL(V)]}(\bb{Q}\{()\}) \to \widetilde{F}^{k-1}_0(V)\]
compatible with augmentations by setting it to be $() \longmapsto (v^\bullet)$.

For the induction step, suppose that $\widetilde{j_V}^{k-1}$ has been defined in degrees $\le p$ for some $p \ge 0$, so that if $[g_1|\compactcdots|g_p]$ is $v^\bullet$-generic then it is sent to $(v^\bullet, g_1 \cdot 
v^\bullet,\compactldots,(g_1\compactcdots g_p) \cdot v^\bullet)$. We need to show that we can produce $(\widetilde{j}_V^{k-1})_{p+1}$ with the correct formula on $v^\bullet$-generic tuples. Since its domain is a free $\GL(V)$-module on $(p+1)$-tuples of elements, it suffices to define it on generators. When $[g_1|\compactcdots|g_{p+1}]$ is $v^\bullet$-generic so are all the terms in the expression for its differential, and thus it is immediate to verify that 
\[(\widetilde{j}_V^{k-1})_{p+1}([g_1|\compactcdots|g_{p+1}])=(v^\bullet, g_1 \cdot 
v^\bullet,\compactldots,(g_1\compactcdots g_{p+1}) \cdot v^\bullet)\]
is a valid extension. On generators that are not $v^\bullet$-generic we choose an arbitrary extension, which is guaranteed to exist by freeness of the domain.  
\end{proof}

We conclude by explicitly recording the construction that yields \cref{thm:extension-criterion}:

\begin{corollary} \label{cor:formulas-chain-to-bigrassmannian}
    Let $m \ge 0$, and fix a partial $(m+1)$-frame $v^\bullet$ in $F^{m+n}$, then the maps 
    \[F_m \colon C_*(\GL_{m+n}(F)) \lra BC_*(n)\]
    can be chosen to be given on a $v^\bullet$-generic element $[g_1|\compactcdots|g_p]$ by the element 
    \[\bigoplus_{j=0}^{m} \sum_{\substack{
i_0+\compactcdots+i_p=j\\
i_0,\compactldots, i_p\ge 0
}}(-1)^j (\!( \langle v_0^1,\compactldots,v_0^{i_0},\compactldots,v_p^1,\compactldots,v_p^{i_p} \rangle| v_0^{i_0+1},\compactldots, v_p^{i_p+1})\!)\]
of $\bigoplus_{j=0}^m {C}_{p}(m+n-j) = BC_p(n)$. Here $v_0^\bullet=v^\bullet$ and $v_i^\bullet \coloneq (g_1\compactcdots g_i) \cdot v^\bullet$ for $1 \le i \le p$, and where we interpret a configuration as $0$ if $j<m+n-p$.  In particular, for $p \ge n$, its $C_p(n)$-component is given by 
\[(-1)^{m}\sum_{\substack{
i_0+\compactcdots+i_p=m\\
i_0,\compactldots, i_p\ge 0
}} (\!( \langle v_0^1,\compactldots,v_0^{i_0},\compactldots,v_p^1,\compactldots,v_p^{i_p} \rangle| v_0^{i_0+1},\compactldots, v_p^{i_p+1})\!) \in {C}_{p}(n) \subset BC_p(n).\]
\end{corollary}

\begin{proof}
By construction, the maps $C_*(\GL_{m+n}(F)) \to BC_*(n)$ are given by the composition 
\[C_*(\GL_{m+n}(F)) \xrightarrow{j_{m+n}^{m+1}} F^m_*(m+n) \xrightarrow{T_*} BD^{m}_*(m+n) \xrightarrow{\phi_n^m} BC_*(n).\]
The formula thus follows by using \cref{prop partial map from group chains to partial frames} (this step forces us to restrict to generic tuples) and \cref{lem map to bi grassmannian complex}. Observe that the symmetrisations in the formula for $T_*$ does not play a role since the spans satisfy $\langle v_1,\compactldots,v_k \rangle= \langle v_{\sigma(1)},\compactldots,v_{\sigma(k)}\rangle$ for any permutation $\sigma \in \fr{S}_k$.
\end{proof}

\bibliographystyle{amsalpha}
\bibliography{./refs}

\bigskip

\end{document}

%% file: preamble.tex
\usepackage{fix-cm}
\usepackage[final]{microtype}
\usepackage[dvipsnames,svgnames,x11names,hyperref]{xcolor}
\usepackage{dsfont,url,graphicx,verbatim,amssymb,enumerate,stmaryrd,booktabs,lmodern,mathtools,mathabx,nicefrac}
\SetSymbolFont{stmry}{bold}{U}{stmry}{m}{n}
\usepackage[pagebackref,colorlinks,citecolor=Mahogany,linkcolor=Mahogany,urlcolor=Mahogany,filecolor=Mahogany]{hyperref}
\usepackage[capitalize]{cleveref}
\usepackage[mathscr]{euscript}
\usepackage[margin=1.33in]{geometry}
\usepackage{tikz,tikz-cd}
\usetikzlibrary{matrix,calc,positioning,arrows,decorations.pathreplacing,patterns,arrows,patterns.meta}
\tikzset{
  snake left/.style={
    rounded corners,
    to path={
      let \p1 = (\tikztostart.east),
          \p2 = (\tikztotarget.west),
          \p3 = ($(\p1)!0.5!(\p2)$),
          \n1 = {8pt} 
      in
      (\p1)
      -- (\x1 + \n1, \y1)
      -- (\x1 + \n1, \y3)
      -- (\x2 - \n1, \y3) \tikztonodes
      -- (\x2 - \n1, \y2)
      -- (\p2)
    }
  }
}

\newtheorem{theorem}{Theorem}[section]
\newtheorem*{theorem*}{Theorem}
\newtheorem{lemma}[theorem]{Lemma}
\newtheorem{proposition}[theorem]{Proposition}
\newtheorem{corollary}[theorem]{Corollary}
\newtheorem*{corollary*}{Corollary}

\newtheorem{atheorem}{Theorem}

\newtheorem{innercustomgeneric}{\customgenericname}
\providecommand{\customgenericname}{}
\newcommand{\newcustomtheorem}[2]{%
  \newenvironment{#1}[1]
  {%
   \ifdefined\crefalias\crefalias{innercustomgeneric}{#2}\fi
   \renewcommand\customgenericname{#2}%
   \renewcommand\theinnercustomgeneric{##1}%
   \innercustomgeneric
  }
  {\endinnercustomgeneric}%
  \ifdefined\crefname\crefname{#2}{#2}{#2s}\fi
}

\newcustomtheorem{customthm}{Theorem}

\newcustomtheorem{customconj}{Conjecture}
\theoremstyle{definition}
\newtheorem{definition}[theorem]{Definition}
\newtheorem*{definition*}{Definition}

\newtheorem{notation}[theorem]{Notation}

\theoremstyle{remark}
\newtheorem{example}[theorem]{Example}
\newtheorem*{example*}{Example}
\newtheorem*{remark*}{Remark}

\newtheorem{remark}[theorem]{Remark}

\renewcommand{\sf}[1]{{\mathsf{#1}}}

\renewcommand{\rm}[1]{{\mathrm{#1}}}
\newcommand{\lra}{\longrightarrow}

\newcommand{\ol}[1]{{\overline{#1}}}

\newcommand{\ds}[1]{{\mathds{#1}}}
\newcommand{\fr}[1]{{\mathfrak{#1}}}
\newcommand{\bb}[1]{{\mathds{#1}}}

\makeatletter
\newcommand{\superimpose}[2]{{%
  \ooalign{%
    \hfil$\m@th#1\@firstoftwo#2$\hfil\cr
    \hfil$\m@th#1\@secondoftwo#2$\hfil\cr
  }%
}}
\makeatother

\newcommand{\GL}{\rm{GL}}

\newcommand{\C}{\ds{C}}
\newcommand{\R}{\ds{R}}

\newcommand{\prim}{\rm{prim}}

\newcommand{\free}{\rm{free}}

\newcommand{\gr}{\rm{gr}}

\newcommand{\SL}{\rm{SL}}

\newcommand{\id}{\rm{id}}
\newcommand{\inc}{\rm{inc}}

\newcommand{\compactldots}{\mathinner{\ldotp\mkern-2mu\ldotp\mkern-2mu\ldotp}}
\newcommand{\compactcdots}{\mathinner{\cdotp\mkern-2mu\cdotp\mkern-2mu\cdotp}}

\DeclareMathSymbol{\shortminus}{\mathbin}{AMSa}{"39}

\newcommand{\circled}[1]{\raisebox{.5pt}{\textcircled{\raisebox{-.9pt} {#1}}}}

\DeclareFontFamily{U}{min}{}
\DeclareFontShape{U}{min}{m}{n}{<-> udmj30}{}

\AtBeginDocument{%
	\def\MR#1{}
}